\documentclass[11pt,reqno]{amsart}
\usepackage{amsmath}
\usepackage{amsthm}
\usepackage[T1]{fontenc}
\usepackage{mathrsfs}
\usepackage{latexsym}
\usepackage{amssymb} 
\usepackage{graphicx}
\usepackage{tikz}
\usepackage{float}
\usepackage{extarrows}
\usepackage{epstopdf}
\usepackage{caption}
\usepackage{subcaption}
\usepackage{xcolor}
\usepackage[a4paper, right=3cm,top=3cm, bottom=3cm]{geometry}
\floatstyle{boxed}
\floatstyle{plain}

\numberwithin{equation}{section}
\newcommand{\R}{\mathbb{R}}
\newcommand{\C}{\mathbb{C}}

\newcommand{\N}{\mathbb{N}}
\newcommand{\re}{\mathrm{Re\,}}

\newcommand{\chookrightarrow}{\mathrel{\lhook\joinrel\relbar\kern-.8ex\joinrel\lhook\joinrel\rightarrow}}

\newcommand{\A}{\mathbb{A}}

\newcommand{\e}{\varepsilon}

\normalsize

\DeclareMathOperator{\Rey}{\textbf{r}}
\DeclareMathOperator{\spec}{spec}

\newtheorem{satz}{Proposition}[section] 
\newtheorem{lem}[satz]{Lemma} 
\newtheorem{thm}[satz]{Theorem}

\newtheorem{cor}[satz]{Corollary}

\definecolor{gray}{gray}{0.50}
\definecolor{lred}{rgb}{1.0,0.5,0.5}
\definecolor{dgreen}{rgb}{0,1,1}
\definecolor{luh-dark-blue}{rgb}{0.0, 0.313, 0.608}

\title[Two--phase thin film model with insoluble surfactant]{Modeling and Analysis of a two--phase thin film model with insoluble surfactant}
\author{Gabriele Bruell}
\address{Institute f\"ur Angewandte Mathematik, Leibniz Universit\"at Hannover, Welfengarten 1, 30167 Hannover, Germany}
\email{bruell@ifam.uni-hannover.de}

\subjclass{35K65, 35B35, 35Q35, 76D08}
\keywords{Two-phase thin film; surfactant; lubrication approximation; degenerate parabolic system; local well-posedness; linearized stability}

\usepackage[colorlinks=true]{hyperref}
\hypersetup{urlcolor=luh-dark-blue, citecolor=lred, linkcolor=luh-dark-blue}

\begin{document}

\begin{abstract} In this paper we consider a two--phase thin film consisting of two immiscible viscous fluids endowed with a layer of insoluble surfactant on the surface of the upper fluid. The governing equations for the two film heights and the surfactant concentration are derived using a lubrication approximation. Taking gravitational forces into account but neglecting capillary effects, the resulting system of evolution equations is parabolic, strongly coupled, of second order and degenerated in the equations for the two film heights. Incorporating on the contrary capillary forces and neglecting the effects of gravitation, the system of evolution equations is parabolic, degenerated and of fourth--order for the film heights, strongly coupled to a second--order transport equation for the surfactant concentration. Local well--posedness and asymptotic stability are shown for both systems.
\end{abstract}

\maketitle

\section{Introduction}
\allowdisplaybreaks
The study of thin film equations constitutes a rich and complex area of research with a long list of contributions by physicists, engineers and mathematicians. Of particular fascination for many scientists is the role of surface tension and the influence of surface active agents (short \emph{surfactants}) on the dynamics of thin liquid films since this finds applications in various industrial and biomedical fields. As for instants in surfactant replacement therapy, coating flow technology or film drainage in emulsions and foams. Surfactants act on the surface of a fluid film by lowering the surface tension and induce a  twofold dynamic. On the one hand, the resulting surface gradients influence the dynamics of the fluid film. On the other hand, the surfactant itself spreads along the interface due to the surface tension gradients. The latter aspect is called \emph{Marangoni effect}.
Pioneering results on the dynamics of a thin fluid with insoluble surfactant are \cite{GG, JG, JG2}, where the approach via \emph{lubrication approximation} for thin liquid films is used and first numerical result are presented under consideration of different driving forces. Although, during the last decades there has been various modeling and numerical treatment of several aspects of the surfactant induced movement of thin films (see e.g. \cite{BE, BGN, BN, CMW, GG, JG, JG2, Mat}), only recently analytical investigations have started. Regarding the one--phase problem with insoluble surfactant, several authors contributed to the analysis of well--posedness and existence of global weak solutions for a coupled system of evolution equations describing the dynamics of the interface and the surfactant spreading under certain assumptions on the driving forces (see \cite{CT, EW1, EW2, GW, Ren} and references therein).  In absence of capillary and intermolecular forces but including gravitational forces, the through lubrication approximation derived system in \cite{EW1} is of second order and local well--posedness as well as asymptotic stability of steady states are proven. In particular, the surfactants in \cite{EW1} are considered to be soluble, which leads to an additional evolution equation for the surfactant distribution in the bulk. Investigating the dynamics of a two--phase thin film flow with insoluble surfactant, we resort not only to results for thin film equations with surfactant, but also to the analytical studies of two--phase thin films. As for instance in \cite{EMM} local well--posedeness and asymptotic stability of a thin--film approximation of the two--phase Stokes problem are investigated by methods of \emph{semigroup theory}  and the \emph{principal of linearized stability}.
 A similar approach has also been successfully applied in \cite{EM13} to prove local existence and stability results for a strongly coupled fourth--order  degenerated  parabolic system modeling the motion of two thin fluid films in the presence of gravity and capillary forces.

In this paper a mathematical model for the evolution of a two--phase flow with insoluble surfactant is presented. The two--phase flow consists of two immiscible, incompressible Newtonian and viscous thin liquid films on top of each other on a solid substrate. We assume that there is no contact angle between the two--phase flow and the bottom, which places the setting in the context of complete wetting. The interface of the upper fluid is endowed with a layer of insoluble surfactant. 
Based on the full Navier--Stokes equation describing the motion of the two viscous fluid films and an advection--transport equation for the spreading of surfactant on the free surface, we derive a system of degenerated strongly coupled parabolic equations for the evolution of the two film heights and the surfactant concentration, by lubrication approximation and cross--sectional averaging.
Depending on the considered driving forces, the evolution equations for the film heights are of fourth order (if capillary forces are taken into account) or of second order (if gravitational forces are considered and capillary effects neglected). Since both systems appear to have a very similar structure, we formulate them together in one set of equations. The system describing the gravity driven flow can then be recovered  by setting $k=1$ and the capillary driven flow by $k=3$ in \eqref{wEE} below. 
Letting $f=f(t,x)$ and $g=g(t,x)$ denote the two film heights and $\Gamma=\Gamma(t,x)$ the surfactant concentration, respectively, at time $t\geq 0$ and position $x\in (0,L)$ the system  
\begin{align} 
\nonumber
&\partial_t f =\partial_x \left[ f\left(\displaystyle{\frac{R_kf^2}{3}}\partial_x^k f + S_k\mu \left(\displaystyle{\frac{f^2}{3}} +\displaystyle{\frac{fg}{2}} \right) \partial_x^k(f+g) -\mu 	\displaystyle{\frac{f}{2}}\partial_x\sigma(\Gamma)\right)\right],\\[15pt]
\label{wEE}
& \partial_t g = \partial_x \left[g\left( \displaystyle{\frac{R_kf^2}{2}} \partial_x^k f +S_k\left(  \displaystyle{\frac{g^2}{3}} +\mu \left( \displaystyle{\frac{f^2}{2}}+fg\right)\right)\partial_x^k (f+g) 
-\left(\mu f +\displaystyle{\frac{g}{2}} \right) \partial_x \sigma(\Gamma)\right)\right], \\[15pt]
\nonumber
&\partial_t \Gamma =\partial_x\left[\Gamma\left( \displaystyle{\frac{R_kf^2}{2}} \partial_x^k f +S_k\left(  \displaystyle{\frac{g^2}{2}}+\mu \left(\displaystyle{\frac{f^2}{2}}+fg\right)\right)\partial_x^k (f+g)
- \left(\mu f+g\right) \partial_x \sigma(\Gamma)\right)+ D\partial_x\Gamma\right]
\end{align}
models the motion of a two--phase flow with insoluble surfactant in the presents of gravitational forces (capillary effects are neglected), when $k=1$, and in the presents of capillary effects (gravitational forces are neglected), when $k=3$. Observe that if $k=1$ all three evolution equations in \eqref{wEE} are of second order whereas in the case $k=3$ the equations for the film heights are of fourth order coupled to a second--order equation for the surfactant concentration.
The function $\sigma=\sigma(\Gamma)$ denotes the surface tension, which depends decreasingly on the surfactant concentration $\Gamma$. The constants $R_k, S_k$, $k=1,3$, are given by
\begin{alignat*}{2}
&R_1:= G_1-G_2\mu,\quad &&S_1:=G_2,\\
&R_3:= -\sigma_1^c,\quad &&S_3:=-\sigma_2^c,
\end{alignat*}
where $G_1= \rho_1 G,G_2=\rho_2G$ with $G$ being a modified gravitational constant. Furthermore, $\rho_1,\rho_2$ and  $\sigma_1^c, \sigma_2^c$ represent the densities and the surface tension coefficients of the lower and the upper fluid, respectively, and $\mu:=\frac{\mu_2}{\mu_1}$ measures the relative viscosity between the fluids, where $\mu_1$ denotes the viscosity of the lower and $\mu_2$  the viscosity of the upper fluid.
The positive constant $D$ represents the diffusion of surfactants along  the surface. 

The outline of the paper is as follows. In Section \ref{PM} we present a derivation of the system of evolution equations \eqref{wEE} by applying lubrication approximation to the governing equations of the motion of a two--phase flow and the surfactant spreading. Moreover, we provide an energy functional for the system \eqref{wEE}, which enables us to determine the set of steady states. In Section \ref{well posedness gravity} a well--posedness result for the gravity driven two--phase flow with insoluble surfactant is proven. Furthermore, we show asymptotic stability for steady states. The corresponding results for the capillary driven flow are established in Section \ref{well posedness capillary}.

The investigation of non--negative global weak solutions to the capillary driven system is subject of a forthcoming paper.

\section{Mathematical model}\label{PM}
We consider two viscous, incompressible Newtonian and immiscible thin films  on top of each other on a horizontal impermeable bottom at $z=0$ with lateral boundaries at $x=0,L$, occupying the regions $\Omega_1$, $\Omega_2$, respectively, with a layer of insoluble surfactant on the surface of the upper fluid. 
We assume the surface tension on the interface separating the fluids to be independent of external influences and the material outside of the two--phase flow to  be static and with zero pressure. Let $L$ be the length of the two--phase film and take the undisturbed film height $H$ to be given as small compared to the film length, that is  $\frac{H}{L} =\varepsilon$ with $\varepsilon \ll 1$.  
By cross--sectional averaging we assume the film to be uniform in one horizontal level and let $x$ and $z$ denote the horizontal and vertical direction, respectively. Further,  we denote the two film heights  by $f$ and $g$, so that the free surfaces at time $t\geq 0$ and position $x\in (0,L)$ are located at $z=f(t,x)$ and $z=(f+g)(t,x)$, see Figure \ref{Fig1}. The concentration of surfactant at time $t\geq 0$ and position $x\in (0,L)$ is given by $\Gamma(t,x)$.

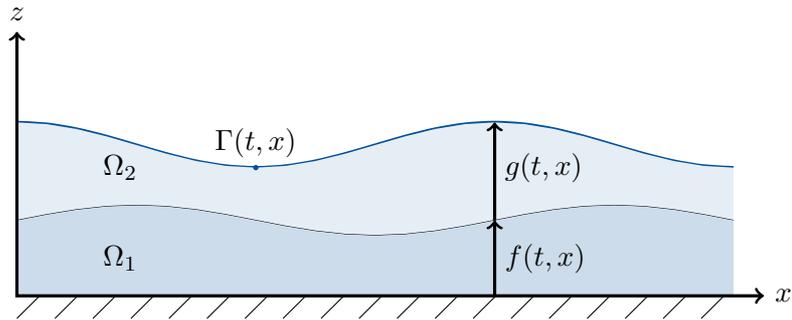
\begin{figure}[h]
\centering
\begin{tikzpicture}[domain=0:3*pi, scale=1] 
\draw[color=black] plot (\x,{0.3*cos(\x r)+2}); 
\draw[very thick, smooth, variable=\x, luh-dark-blue] plot (\x,{0.3*cos(\x r)+2}); 
\fill[luh-dark-blue!10] plot[domain=0:3*pi] (\x,{0.2*sin(\x r)+1}) -- plot[domain=3*pi:0] (\x,{0.3*cos(\x r)+2});
\draw[color=black] plot (\x,{0.2*sin(\x r)+1});
\fill[luh-dark-blue!20] plot[domain=0:3*pi] (\x,0) -- plot[domain=3*pi:0] (\x,{0.2*sin(\x r)+1});
\draw[very thick,<->] (3*pi+0.4,0) node[right] {$x$} -- (0,0) -- (0,3.5) node[above] {$z$};
\draw[very thick,->] (2*pi,1) -- (2*pi,2.3);
\node[right] at (2*pi,1.7) {$g(t,x)$};
\draw[very thick,->] (2*pi,0) -- (2*pi,1);
\node[right] at (2*pi,0.5) {$f(t,x)$};
\coordinate[label=above:{$\Gamma(t,x)$}] (A) at (pi,1.7);
\fill[color=luh-dark-blue] (A) circle (1pt);
\node[right] at (1,0.5) {$\Omega_1$};
\node[right] at (1,1.7) {$\Omega_2$};
\draw[-] (0,-0.3) -- (0.3, 0);
\draw[-] (0.5,-0.3) -- +(0.3, 0.3);
\draw[-] (1,-0.3) -- +(0.3, 0.3);
\draw[-] (1.5,-0.3) -- +(0.3, 0.3);
\draw[-] (2,-0.3) -- +(0.3, 0.3);
\draw[-] (2.5,-0.3) -- +(0.3, 0.3);
\draw[-] (3,-0.3) -- +(0.3, 0.3);
\draw[-] (3.5,-0.3) -- +(0.3, 0.3);
\draw[-] (4,-0.3) -- +(0.3, 0.3);
\draw[-] (4.5,-0.3) -- +(0.3, 0.3);
\draw[-] (5,-0.3) -- +(0.3, 0.3);
\draw[-] (5.5,-0.3) -- +(0.3, 0.3);
\draw[-] (6,-0.3) -- +(0.3, 0.3);
\draw[-] (6.5,-0.3) -- +(0.3, 0.3);
\draw[-] (7,-0.3) -- +(0.3, 0.3);
\draw[-] (7.5,-0.3) -- +(0.3, 0.3);
\draw[-] (8,-0.3) -- +(0.3, 0.3);
\draw[-] (8.5,-0.3) -- +(0.3, 0.3);
\draw[-] (9,-0.3) -- +(0.3, 0.3);
\end{tikzpicture} 
\caption{Scheme of the two--phase thin film flow with insoluble surfactant}\label{Fig1}
\end{figure}

As common in the analysis of thin films (see e.g. \cite{GG, JG, JG2} for pioneering works), we apply a lubrication approximation to the governing equations for the dynamics of the fluids and the surfactant concentration together with suitable boundary conditions, in order to derive the system of evolution equations \eqref{wEE} for the two film heights $f,g$ and the concentration of surfactant $\Gamma$ on the fluid--gas interface. 
Setting $i=1,2$, the velocity field of the fluid contained in $\Omega_i$ will be denoted by $v_i=(u_i,w_i)$, where each particle of the fluid contained in $\Omega_i$  is moving with the velocity $u_i(t,x,z)$ in horizontal and $w_i(t,x,z)$ in vertical direction. The velocity and the pressure, given by $p_i$, are functions of position and time. The gravitational acceleration is given by $\gamma=(0,G)$. Moreover, assuming the fluid to be incompressible and Newtonian, the density and viscosity of the fluids, denoted by $\rho_i$ and $\mu_i$ are material constants. 

The governing equations for the motion of a viscous, incompressible and Newtonian fluid occupying $\Omega_i,i=1,2,$ is given by the \emph{Navier--Stokes equation}
\begin{equation}\label{Ns}
\rho_i(\partial_t v_i+ (v_i \cdot \nabla )v_i ) = \mu_i \Delta v_i - \nabla p_i - \rho_i \gamma.
\end{equation}
Further, conservation of mass for incompressible fluids implies the \emph{continuity equation}
\begin{equation}\label{Ce}
\partial_x u_i +\partial_z w_i =0 
\end{equation}
in $\Omega_i, \,i=1,2$.
The dynamics of thin liquid films is strongly influenced by surface tension (cf. e.g. \cite{Lam}). Since surface tension affects only the free surface, it does not appear in the Navier--Stokes equations, but contributes to the motion of a fluid through boundary conditions.
The surfactant spreading on the free surface $z=f+g$ is governed by the \emph{advection--transport equation}
\begin{equation}\label{Ss}
\partial_t \Gamma +\partial_x(u_2\Gamma-D\partial_x\Gamma)=0,
\end{equation}
where $D>0$ is the surface diffusion coefficient. 
Note that additionally to the diffusion, the spreading of surfactant is also induced by surface tension gradients, which  occur due to the present of surfactant itself (Marangoni effect). This effect will enter into the tangential balance equation \eqref{tbC}.
Since the bottom $z=0$ is impermeable, there is no transfer across this boundary and the perpendicular velocity at the bottom is zero. Furthermore, we suppose a \emph{no--slip boundary condition} on $z=0$. Assuming the velocity field to be continuous across the immiscible fluid--fluid interface $z=f$ and that there is no diffusion between the fluids, the velocity fields $v_1$ and $v_2$ equal at $z=f$. We have
\begin{alignat}{2}
\label{BB}v_1 &=0\qquad &&\mbox{on}\quad z=0,\\
\label{nos2}v_1&=v_2\qquad &&\mbox{on}\quad z=f.
\end{alignat}
Due to interfacial tension, the \emph{stress balance equations}
\begin{equation*}
\left\{ \begin{array}{lcl}
 [\Sigma(v_1,p_1)-\Sigma(v_2,p_2)]n_1&=& \sigma_1 \kappa_f n_1+ \nabla_s \sigma_1 \\[5pt]
 \Sigma(v_2,p_2)n_2&=& \sigma_2 \kappa_2 n_2+ \nabla_s \sigma_2 
\end{array}\right. \qquad \begin{array}{lcl}&\mbox{on}& \; z=f,\\[5pt]  &\mbox{on}& \; z=f+g, \end{array}
 \end{equation*}
where $\Sigma(v_i,p_i) = \frac{1}{2}\mu_i(\nabla v_i + \nabla v_i^{T})- p_i $ denotes the stress tensor, $\kappa_i$ the mean curvature and $\sigma_i$ the surface tension of the interface being the upper boundary of the domain $\Omega_i$ and $\nabla_s \sigma_i$ the gradient of $\sigma_i$ in direction of the surface. Here, $n_i$ denotes the outher normal pointing outwards of $\Omega_i$, $i=1,2$. Multiplying the stress balance equation by $n_i$, yields the \emph{normal stress balance equation}
\begin{equation}\label{nbC} \left\{ \begin{array}{lcl} 
([\Sigma(v_1,p_1)-\Sigma(v_2,p_2)]n_1)\cdot n_1 &=& \sigma_1\kappa_1 \\[5pt]
 (\Sigma(v_2,p_2) n_2) \cdot n_2 &=& \sigma_2\kappa_2
 \end{array}\right.\qquad \begin{array}{lcl}&\mbox{on}& \; z=f,\\[5pt]  &\mbox{on}& \; z=f+g, \end{array}
 \end{equation}
 where the mean curvatures $\kappa_1, \kappa_2$ are  given by 
\begin{align*}
\kappa_1 = \frac{\partial_x^2 f}{(1+ |\partial_x f|^2)^{\frac{3}{2}}},\qquad \kappa_2 = \frac{\partial_x^2 (f+g)}{(1+ |\partial_x (f+g)|^2)^{\frac{3}{2}}}.
\end{align*}
The surface tension coefficient $\sigma_1$ on $z=f$ is constant, whereas the surface tension of the free surface of the upper fluid depends non--increasingly on the surfactant concentration $\sigma_2=\sigma_2(\Gamma)$. 
Thus,  multiplying the stress balance equation by the tangential vector $t_i$, leads to the \emph{tangential stress balance equation}
\begin{equation}\label{tbC}\left\{\begin{array}{lcl}
 ([\Sigma(v_1,p_1)-\Sigma(v_2,p_2)]n_1) \cdot t_1 &=& 0 \\[5pt]
 (\Sigma(v_2,p_2) n_2) \cdot t_2 &=&  \displaystyle{\frac{\partial_x \sigma_2(\Gamma)}{\sqrt{1+|\partial_x^2 (f+g)|}}}
 \end{array}\right.\qquad \begin{array}{lcl}&\mbox{on}& \; z=f,\\[5pt]  &\mbox{on}& \; z=f+g, \end{array}
 \end{equation} 
 where 
 \begin{align*}
 \nabla_s \sigma_2\cdot t_2 =\displaystyle{\frac{\partial_x \sigma_2(\Gamma)}{\sqrt{1+|\partial_x^2 (f+g)|}}}.
 \end{align*}
Observe that the normal stress balance is controlled by the capillary forces, whereas the Marangoni forces, induced by the surfactant, enter the tangential stress balance equation. Furthermore, we prescribe a \emph{kinematic boundary condition} on the interfaces located at $z=h_i$, $i=1,2$,
\begin{align}
\label{kin2}\partial_t h_i + u_2 \partial_x h_i = w_2 \qquad \mbox{on} \quad z=h_i,
\end{align}
where $h_1:=f$ and $h_2:= f+g$. 
We summarize that the motion of the two--phase thin film flow with insoluble surfactant is described by the Navier--Stokes equation \eqref{Ns} together with the continuity equation \eqref{Ce}, the surfactant spreading equation \eqref{Ss} and the boundary conditions \eqref{BB}--\eqref{kin2}.

\subsection{Lubrication Approximation}
In the frame of thin film equations, the method of lubrication approximation (cf. e.g. \cite{GG, JG, JG2}) enables to simplify the system of equations by rescaling the parameters and considering the system in the limit, where the relative film height $\frac{H}{L}=\varepsilon$ tends to zero. The obtained equations do not represent the complex mechanisms of the original problem completely, but still preserve the main features.
According to \cite{GG,JG} an appropriate scaling for the lubrication approximation of thin films with surfactant is given by 
\[\bar{x} = \frac{x}{L},\quad \bar{z}= \frac{z}{H}, \quad\bar{t}= \varepsilon^3 \tau_0 t,\]
with unit of time $\tau_0=\frac{1}{s}$, and rescaled functions
\begin{equation*}
\bar{f}(\bar{t},\bar{x})=\frac{1}{H}f(t,x),\quad \bar{g}(\bar{t},\bar{x})=\frac{1}{H}g(t,x)
\end{equation*}
as well as
\begin{alignat*}{2}
&u(t,x,z) = \e^3\tau_0L \bar{u}(\bar{t},\bar{x},\bar{z}),\qquad &&w(t,x,z) = \e^4\tau_0L \bar{w}(\bar{t},\bar{x},\bar{z}), \\[5pt] 
& p(t,x,z) = \e\tau_0\mu \bar{p}(\bar{t},\bar{x},\bar{z}),\qquad &&\Gamma(t,x)= \Gamma_m \bar{\Gamma}(\bar{t},\bar{x}), 
\end{alignat*}
where $\Gamma_m$ represents the critical micelle concentration. As suggested by \cite{GG, JG} we choose the scaling $D= \e^3\tau_0L\bar{D}$ and
\begin{equation*}
\sigma_1= \mu_1\tau_0 L \sigma_1^c,\qquad \sigma_2=\mu_2\tau_0L\left(\sigma_2^c+\e^2\bar{\sigma}\right),
\end{equation*}
where $\sigma_2^c$ is the rescaled surface tension coefficient of the interface when $\bar{\Gamma}=\Gamma_m$ and $\bar{\sigma}$ the part of the surface tension coefficient, which depends on the the surfactant concentration. Recall that $\sigma_1= \mu_1\tau_0 L \sigma_1^c$ is constant, since the surface tension coefficient of the interface between the fluids is independent of $\Gamma$, for the insoluble surfactant is acting on the surface of the upper fluid only. Eventually, the fluids are contained in the regions $\bar{\Omega}_1:=\{0\leq \bar{z}\leq \bar{f}\}$ and $\bar{\Omega}_2:=\{\bar{f}\leq \bar{z}\leq \bar{f}+\bar{g}\}$. It can be easily checked that the rescaled set of equations describing the motion of the two--phase flow with insoluble surfactant is given by
 \begin{alignat}{2}
 \label{NS}& \hspace{-0.5cm} 
 \left\{\begin{array}{lll}
 \e^2 \Rey \left( \partial_{\bar{t}} \bar{u}_i +\bar{u}_i \partial_{\bar{x}} \bar{u}_i +  \bar{w}_i \partial_{\bar{z}} \bar{u}_i\right)&=& \left(\varepsilon^2\partial_{\bar{x}}^2 u_i +\partial_{\bar{z}}^2 \bar{u}_i\right) - \partial_{\bar{x}} \bar{p}_i, \\[5pt]
 \e^5 \Rey\left(\partial_{\bar{t}} \bar{w}_i + \bar{u}_i \partial_{\bar{x}} \bar{w}_i + \bar{w}_i \partial_{\bar{z}} \bar{w}_i\right)&=& \e^2\left(\e^2\partial_{\bar{x}}^2 \bar{w}_i +\partial_{\bar{z}}^2 \bar{w}_i\right) - \partial_{\bar{z}} \bar{p}_i -\displaystyle{\frac{\rho_i LG}{\mu_i \tau_0}}
 \end{array} \right.
  &&\mbox{in}\; \bar{\Omega}_i,  \\[5pt]
\label{C}& \; \partial_{\bar{x}} \bar{u}_i +\partial_{\bar{z}} \bar{w}_i =0\quad &&\;\mbox{in}\; \bar{\Omega}_i,  \\[5pt]
\label{MN}& \hspace{-0.5cm} 
\left\{\begin{array}{lll}
\bar{w}_1=\bar{u}_1 =0 \\[5pt]
\bar{w}_1=\bar{w}_2,\quad \bar{u}_1=\bar{u}_2 
\end{array}\right.
&& \begin{array}{lll} z=0 \\[10pt] \bar{z}=\bar{f}, \end{array} 
\\[10pt]
\label{KBC}&
\;\partial_{\bar{t}} \bar{h}_i + \bar{u}_i \partial_{\bar{x}} \bar{h}_i = \bar{w}_i &&\; \bar{z}=\bar{h}_i, \\[5pt]
\label{NBC}&\hspace{-0.5cm} 
 \left\{ \begin{array}{lcl} 
 ([\Sigma(\bar{v}_1,\bar{p}_1)-\Sigma(\bar{v}_2,\bar{p}_2)]\bar{n}_1)\cdot \bar{n}_1 &=& \frac{\mu_1\tau_0 \sigma_1^c\partial_{\bar{x}}^2\bar{f}\e}{\sqrt{1-|\partial_{\bar{x}}\bar{f}\e|^2}}, \\[5pt]
  (\Sigma(\bar{v}_2,\bar{p}_2) \bar{n}_2) \cdot \bar{n}_2 &=& \frac{\mu_2\tau_0 \sigma_2^c\partial_{\bar{x}}^2(\bar{f}+\bar{g})\e}{\sqrt{1-|\partial_{\bar{x}}(\bar{f}+\bar{g})\e|^2}}, 
  \end{array}\right.
&& \begin{array}{lll} \bar{z}=\bar{f} \\[10pt] \bar{z}=\bar{f}+\bar{g}, \end{array} 
\\[5pt]
\label{TBC}&\hspace{-0.5cm} 
\left\{ \begin{array}{lcl} 
([\Sigma(\bar{v}_1,\bar{p}_1)-\Sigma(\bar{v}_2,\bar{p}_2)]\bar{t}_1)\cdot \bar{t}_1 &=& 0,  \\[5pt]
 (\Sigma(\bar{v}_2,\bar{p}_2) \bar{t}_2) \cdot \bar{t}_2 &=& \frac{\mu_2\tau_0\e^2\partial_{\bar{x}}\bar{\sigma}(\bar{\Gamma})}{\sqrt{1-|\partial_{\bar{x}}(\bar{f}+\bar{g})\e|^2}}, 
 \end{array}\right.
&& \begin{array}{lll} \bar{z}=\bar{f} \\[5pt] \bar{z}=\bar{f}+\bar{g}, \end{array} 
\\[5pt]
\label{S}&\; \partial_{\bar{t}} \bar{\Gamma} + \partial_{\bar{x}}(\bar{u}_2 \bar{\Gamma} - D \partial_{\bar{x}} \bar{\Gamma}) =0 &&\; \bar{z}=\bar{f}+\bar{g},
  \end{alignat}
where $\bar{n}_i$, $\bar{t}_i$ are the rescaled outer normal and tangential vectors, respectively, and
\[\Sigma(\bar{v}_i,\bar{p}_i)=\mu_i \tau_0 \e \begin{pmatrix} 2\e^2 \partial_{\bar{x}}\bar{u}_i-\frac{\bar{p}_i}{\mu_i}  & \e^3 \partial_{\bar{x}}\bar{w}_i+\e\partial_{\bar{z}}\bar{u}_i \\
																\e^3 \partial_{\bar{x}}\bar{w}_i+\e \partial_{\bar{z}}\bar{u}_i & 2 \e^3 \partial_{\bar{x}}\bar{w}_i-\frac{\bar{p}_i}{\mu_i}\end{pmatrix}\]
is the stress tensor at the free surface located at $\bar{h}_i$ with respect to the rescaled variables. Moreover, $\Rey := \frac{\rho_i \tau_0^2 \e^3 L^2}{\mu_i}$ in \eqref{NS} is the so--called \emph{Reynold's number}, which is the ratio of inertial forces to viscous forces and characterizes whether the flow is \emph{laminar} (small Reynold's number) or \emph{turbulent} (high Reynold's number).
Observe that the lubrication approximation does not affect the continuity equation \eqref{C}, nor the conservation of mass and no--slip condition \eqref{MN}, the kinematic boundary condition \eqref{KBC} or the equation for the surfactant spreading \eqref{S}. However, the Navier--Stokes \eqref{NS} and the stress balance equations \eqref{NBC}, \eqref{TBC} reduce under lubrication approximation ($\e\rightarrow 0$) to
\begin{alignat}{2}
 \label{NS2}&  
 \left\{\begin{array}{lll}
\partial_{\bar{x}} \bar{p}_i+\partial_{\bar{x}}^2\bar{u}_i&=& 0, \\[5pt]
  \partial_{\bar{z}} \bar{p}-G_i&=& 0
 \end{array} \right.
  \quad &&\;\mbox{in}\; \bar{\Omega}_i,  \\[5pt]
\label{NBC2}&
 \left\{ \begin{array}{lcl} 
-\bar{p}_1 +\mu \bar{p}_2&=& \sigma_1^c\partial_{\bar{x}}^2\bar{f}, \\[5pt]
 -\bar{p}_2 &=& \sigma_2^c\partial_{\bar{x}}^2(\bar{f}+\bar{g}), 
  \end{array}\right.
&& \begin{array}{lll} \bar{z}=\bar{f} \\[10pt] \bar{z}=\bar{f}+\bar{g}, \end{array} 
\\[5pt]
\label{TBC2}&
\left\{ \begin{array}{lcl} 
\partial_{\bar{z}}\bar{u}_1  &=& \mu \partial_{\bar{z}}\bar{u}_2 ,  \\[5pt]
\partial_{\bar{z}}\bar{u}_2 &=&\partial_{\bar{x}}\bar{\sigma}(\bar{\Gamma}), 
 \end{array}\right.
&& \begin{array}{lll} \bar{z}=\bar{f} \\[5pt] \bar{z}=\bar{f}+\bar{g}, \end{array} 
  \end{alignat}
where 
\begin{equation}\label{Gi}
G_i:= \frac{\rho_iL}{\mu_i \tau_0}G,\qquad i=1,2,\qquad \mu:=\frac{\mu_2}{\mu_1}
\end{equation} are a modified gravitational constant depending on the density and viscosity of the fluid and the relative viscosity, respectively.

\subsection{Evolution equations}

Similar as for instance in \cite{EW1}, we use \eqref{C}--\eqref{KBC} and \eqref{S}--\eqref{TBC2} in order to derive evolution equations for the two film heights $f$, $g$ and the concentration of surfactant $\Gamma$. In order to simplify notations, we will skip the bar.

Integrating $(\ref{NS2})$  with respect to $z$ and using $(\ref{NBC2})$ we obtain equations for the pressure within the fluids contained in $\Omega_i,  i=1,2$,
\begin{align}
\label{AA}p_1(t,x,z)&= G_1(f(t,x)-z)=\mu p_2(t,x,f)-\sigma_1^c\partial_x^2f(t,x), \\[5pt]
\label{B}p_2(t,x,z)&= G_2(f(t,x)+g(t,x)-z)= \sigma_2^c \partial_x^2 (f+g)(t,x).
\end{align}
Plugging equation $(\ref{B})$ into $(\ref{AA})$, the pressure within the lower fluid is given by
\[p_1(t,x,z)=G_1(f(t,x)-z)= G_2\mu g(t,x) - \sigma_2^c\mu\partial_x^2(f+g)(t,x)-\sigma_1^c\partial_x^2f(t,x).\]
Differentiating with respect to $x$ and using $(\ref{NS})$ implies
\[-\partial_z^2 u_1(t,x,z)= G_1 \partial_x f(t,x)+ G_2\mu \partial_x g(t,x)- \sigma_2^c\mu\partial_x^3(f+g)(t,x)-\sigma_1^c\partial_x^3f(t,x),\]
hence, by $(\ref{TBC2})$,
\begin{align*}\partial_z u_1(t,x,z) = &- \left( G_1 \partial_x f(t,x) + G_2\mu \partial_x g(t,x)- \sigma_2^c\mu\partial_x^3(f+g)(t,x)-\sigma_1^c\partial_x^3f(t,x)\right)(f(t,x)-z)\\[5pt]
&+\displaystyle{\mu }\partial_z u_2(t,x,f).
\end{align*}
Integrating with respect to $z$ yields, in view of the no--slip boundary condition \eqref{MN},
\begin{align*}
\begin{split}
 u_1(t,x,z) = &- \left( G_1 \partial_x f(t,x) + G_2\mu \partial_x g(t,x) -\sigma_2^c\mu\partial_x^3(f+g)(t,x)-\sigma_1^c\partial_x^3f(t,x)\right)\\[5pt]
&\quad\times \left(f(t,x)z-\frac{1}{2}z^2\right)+\displaystyle{\mu }\partial_z u_2(t,x,f)z.
\end{split}
\end{align*}
Note that
\[\int_0^{f(t,x)} \partial_x u_1(t,x,z)\,dz = -w_1(t,x,z)=-\partial_t f(t,x)-u_1(t,x,f)\partial_x f(t,x),\]
by $(\ref{C})$--$(\ref{KBC})$. Thus
\[\partial_t f(t,x)+\partial_x \left(\int_0^{f(t,x)}u_1(t,x,z)\,dz\right)=0,\]
which is equivalent to
\begin{align}\label{D}
\begin{split}\hspace{-0.3cm}
\partial_t f(t,x)-\partial_x \displaystyle{\int}_0^{f(t,x)}&\Big\{\left( G_1 \partial_x f(t,x) +  G_2\mu \partial_x g(t,x)-  \sigma_2^c\mu\partial_x^3(f+g)(t,x)-\sigma_1^c\partial_x^3f(t,x)\right)\\[5pt] 
&\quad\times\left(f(t,x)z-\frac{1}{2}z^2\right)-\displaystyle{\mu }\partial_z u_2(t,x,f)z\Big\}\,dz=0.
\end{split}
\end{align}
In order to obtain an evolution equation for $f$, which depends only on $g,\Gamma$ and $f$ itself we need to determine an equation for $u_2$. Recalling $(\ref{B})$ and using $(\ref{NS}),(\ref{TBC2})$, we get
\begin{equation}\label{E}\hspace{-0.2cm}
\partial_z u_2(t,x,z)=-\left( G_2\partial_x(f+g)(t,x)-\sigma_2^c\partial_x^3(f+g)(t,x)\right)(f+g-z)(t,x)+\partial_x\sigma_2(\Gamma(t,x)).
\end{equation}
Hence, $(\ref{D})$ and $(\ref{E})$ imply that
\begin{align}\label{fa}
\begin{split}
\partial_t f =\partial_x &\left[ f\left( (G_1-G_2\mu)\displaystyle{\frac{f^2}{3}} \partial_x f +G_2\mu \left(\displaystyle{\frac{f^2}{3}}+\displaystyle{\frac{fg}{2}}\right)\partial_x (f+g) -\mu\displaystyle{\frac{f}{2}}\partial_x \sigma(\Gamma)\right.\right.\\[5pt]
&\qquad\left.\left. -\sigma_1^c\displaystyle{\frac{f^2}{3}}\partial_x^3f- \sigma_2^c\mu\left(\displaystyle{\frac{f^2}{3}}+\displaystyle{\frac{fg}{2}}\right)\partial_x^3 (f+g)\right)\right],
\end{split}
\end{align}
where $f$, $g$ and $\Gamma$ depend on $(t,x)\in (0,\infty)\times (0,L)$.

Owing to \eqref{MN}, \eqref{C} and \eqref{E}, we obtain that
\begin{align}\begin{split}\label{G} \hspace{-0.3cm}
u_2(t,x,z) = &-\Big( G_2\partial_x(f+g)(t,x)-\sigma_2^c\partial_x^3(f+g)(t,x)\Big)\Big[ (f+g)(t,x)z-\displaystyle\frac{1}{2}z^2-\displaystyle\frac{1}{2}f^2(t,x)\\[5pt]
 &-fg(t,x)\Big] +\partial_x\sigma_2(\Gamma(t,x))[z-f(t,x)]\\[5pt]
&-\Big( G_1\partial_xf(t,x)+G_2\mu \partial_x g(t,x)- \sigma_2^c\mu\partial_x^3(f+g)(t,x)-\sigma_1^c\partial_x^3f(t,x)\Big)\displaystyle\frac{f^2(t,x)}{2} \\[5pt]
& -\mu\Big( G_2\partial_x(f+g)(t,x)-\sigma_2^c\partial_x^3(f+g)(t,x)\Big)(fg)(t,x) +\mu\partial_x\sigma_2(\Gamma(t,x))f(t,x).
\end{split}
\end{align}
Hence, in virtue of \eqref{MN}, \eqref{KBC}, the evolution equation for $g$ is determined by
\[\partial_t g(t,x)+\partial_x \left(\int_{f(t,x)}^{(f+g)(t,x)}u_2(t,x,z)\,dz\right)=0\]
and it follows from $(\ref{G})$ that
\begin{align}\label{ga}
\begin{split}
\partial_t g = &\partial_x \left[ g\left((G_1-G_2\mu)\displaystyle\frac{f^2}{2}\partial_x f + \left(G_2 \displaystyle\frac{g^2}{3}+G_2\mu\left(\displaystyle\frac{f^2}{2}+ fg\right)\right)\partial_x(f+g)\right.\right. \\[5pt]
&\; \left.\left. -\left(\mu f+\displaystyle\frac{g}{2}\right)\partial_x\sigma(\Gamma)-\sigma^c_1\displaystyle\frac{f^2}{2}\partial_x^3f - \left(\sigma_2^c \displaystyle\frac{g^2}{3}+\sigma_2^c\mu\left(\displaystyle\frac{f^2}{2}+fg\right)\right)\partial_x^3(f+g) \right)\right],
\end{split}
\end{align}
where $f$, $g$ and $\Gamma$ depend on $(t,x)\in (0,\infty)\times (0,L)$.

The equation for surfactant spreading on the layer $z=f+g$ is given by  the advection--transport equation \eqref{S}
\[\partial_t \Gamma + \partial_x(u_2 \Gamma - D \partial_x \Gamma) =0.\]
In view of $(\ref{G})$ we obtain the following equation for the evolution of $\Gamma$:
\begin{equation}\label{gammaa} \hspace{-0.3cm}
\begin{split}
&\partial_t \Gamma =\partial_x\left[\Gamma\left((G_1-G_2\mu)\displaystyle\frac{f^2}{2}\partial_x f+\left(G_2\displaystyle\frac{g^2}{2}+G_2\mu\left(\displaystyle\frac{f^2}{2}+fg\right)\right)\partial_x(f+g)\right.\right.\\[5pt]
&\;\left.\left. -\left(\mu f +g\right) \partial_x\sigma(\Gamma)- \sigma^c_1\displaystyle\frac{f^2}{2} \partial_x^3 f -\left(\sigma_2^c \displaystyle\frac{g^2}{2}+\sigma_2^c\mu\left(\displaystyle\frac{f^2}{2}+fg\right)\right)\partial_x ^3(f+g) \right)+D\partial_x\Gamma\right],
\end{split}
\end{equation}
where $f$, $g$ and $\Gamma$ depend on $(t,x)\in (0,\infty)\times (0,L)$.

Recalling \eqref{fa}, \eqref{ga} and \eqref{gammaa}, the via lubrication approximation derived system describing the evolution of a two--phase flow driven by gravitational forces only ($k=1$) or capillary effects only ($k=3$) is given by a \emph{strongly coupled, degenerated} system of second ($k=1$) or fourth order ($k=3$), respectively:
\begin{subequations}\label{system}
\begin{align} 
\nonumber
&\hspace{-0.3cm}\partial_t f =\partial_x \left[ f\left(\displaystyle{\frac{R_kf^2}{3}}\partial_x^k f + S_k\mu \left(\displaystyle{\frac{f^2}{3}} +\displaystyle{\frac{fg}{2}} \right) \partial_x^k(f+g) -\mu 	\displaystyle{\frac{f}{2}}\partial_x\sigma(\Gamma)\right)\right],\\[15pt]
\label{EE}
&\hspace{-0.3cm} \partial_t g = \partial_x \left[g\left( \displaystyle{\frac{R_kf^2}{2}} \partial_x^k f +S_k\left(  \displaystyle{\frac{g^2}{3}} +\mu \left( \displaystyle{\frac{f^2}{2}}+fg\right)\right)\partial_x^k (f+g) 
-\left(\mu f +\displaystyle{\frac{g}{2}} \right) \partial_x \sigma(\Gamma)\right)\right], \\[15pt]
\nonumber
&\hspace{-0.3cm}\partial_t \Gamma =\partial_x\left[\Gamma\left( \displaystyle{\frac{R_kf^2}{2}} \partial_x^k f +S_k\left(  \displaystyle{\frac{g^2}{2}}+\mu \left(\displaystyle{\frac{f^2}{2}}+fg\right)\right)\partial_x^k (f+g)
- \left(\mu f+g\right) \partial_x \sigma(\Gamma)\right)+ D\partial_x\Gamma\right]
\end{align}
for $t>0$ and $x\in (0,L)$ with initial data at $t=0$
\begin{equation} \label{ID}
f(0,\cdot)=f^0,\quad g(0,\cdot)=g^0, \quad \Gamma(0,\cdot)=\Gamma^0.
\end{equation}
Furthermore, we impose boundary conditions
\begin{align}\label{NB}
\begin{split}
&\partial_xf=\partial_xg=\partial_x\Gamma =0,\\[5pt]
&\partial_x^kf=\partial_x^kg=0.
\end{split}
\end{align}
at $x=0,L$.
\end{subequations}
The constants $R_k$ and $S_k$ are given by
\begin{alignat*}{2}
&R_1:= G_1-G_2\mu,\quad &&S_1:=G_2,\\
&R_3:= -\sigma_1^c,\quad &&S_3:=-\sigma_2^c.
\end{alignat*}
The degeneracy occurs in the equations for $f$ and $g$ in the sense that if $f$ or $g$ become zero in the first or second equation of \eqref{EE}, respectively, the highest order terms vanish. Hence, the system \eqref{EE} is not uniformly parabolic. It is said to be strongly coupled, since each equation contains highest order derivatives of all three unknowns. 
Observe that due to the special structure of \eqref{EE}, the boundary conditions \eqref{NB} guarantee that the mass of the each fluid and the mass of surfactant concentration is preserved.

\subsection{Energy Functional} If a  solution to \eqref{system} possesses sufficient regularity, there exists an \emph{energy functional} for the system of evolution equations, which provides not only a--priori estimates for $f,g,\Gamma$ and their spatial gradients but determines in particular  steady state solutions of \eqref{system}. The analysis of the asymptotic behavior of steady states will be a subject of the sequel sections.

\begin{lem}\label{EnF} The functionals
\begin{align*}
&\mathcal{E}_k(u):=\displaystyle{\int_0^L}\left\{ \frac{1}{2}\left( R_k f^2+S_k\mu(f+g)^2\right)+\mu\Phi(\Gamma)\right\}\,dx ,\qquad \mbox{for}\quad k=1,\\[10pt]
&\mathcal{E}_k(u):=\displaystyle{\int_0^L}\left\{ -\frac{1}{2}\left( R_k |\partial_xf|^2+S_k\mu|\partial_x(f+g)|^2\right)+\mu\Phi(\Gamma)\right\}\,dx ,\qquad \mbox{for}\quad k=3
\end{align*}
dissipate along sufficient regular  solutions   $u=(f,g,\Gamma)$ to \eqref{system}, where the function $\Phi$ be such that
\[\Phi^{\prime\prime}(s)s=-\sigma(s)\geq 0,\qquad s>0.\]
\end{lem} 

\begin{proof}
 Let $u=(f,g,\Gamma)$ be a solution to \eqref{system} satisfying the regularity 
 \begin{align*}
 (f,g,\Gamma) \in \left\{ \begin{array}{lcl} C^1((0,T); L_2(0,L,\R^3)),\qquad &\mbox{if}&\quad  k=1,\\[5pt] 
											C^1((0,T);H^1(0,L;\R^2))\cap C^1([0,T];L_2(0,L;\R)),\qquad &\mbox{if}&\quad  k=3,\end{array}\right.
 \end{align*}
 where $T\in (0,\infty]$ is the maximal time of existence. Recalling the special sturcture of \eqref{EE}, integration by parts yields \footnote{Note that the boundary terms vanish due to the boundary conditions \eqref{NB}.}
\allowdisplaybreaks
\begin{align*}
&\frac{d}{dt}\mathcal{E}_k(u)  = -\int_0^L \left\{R_k\partial_x^k f\left[ \displaystyle{\frac{R_kf^3}{3}}\partial_x^k f + S_k\mu \left(\displaystyle{\frac{f^3}{3}} +\displaystyle{\frac{f^2g}{2}} \right) \partial_x^k(f+g) -\mu 	\displaystyle{\frac{f^2}{2}}\partial_x\sigma(\Gamma)\right] \right\} \,dx \\[5pt]
& \quad -\int_0^L\left\{S_k\mu\partial_x^k(f+g) \left[ \displaystyle{\frac{R_kf^3}{3}}\partial_x^k f + S_k\mu \left(\displaystyle{\frac{f^3}{3}} +\displaystyle{\frac{f^2g}{2}} \right) \partial_x^k(f+g) -\mu 	\displaystyle{\frac{f^2}{2}}\partial_x\sigma(\Gamma)\right.\right.\\[5pt]
 &\quad\qquad \left.\left. + \displaystyle{\frac{R_kf^2g}{2}} \partial_x^k f +S_k\left(  \displaystyle{\frac{g^3}{3}} +\mu \left( \displaystyle{\frac{f^2g}{2}}+fg^2\right)\right)\partial_x^k (f+g) 
 -\left(\mu fg +\displaystyle{\frac{g^2}{2}} \right) \partial_x \sigma(\Gamma)\right]\right\}\,dx \\[5pt]
& \quad - \int_0^L\left\{\mu\Phi^{\prime\prime}(\Gamma)\partial_x \Gamma \left[\Gamma\left( \displaystyle{\frac{R_kf^2}{2}} \partial_x^k f +S_k\left(  \displaystyle{\frac{g^2}{2}}+\mu \left(\displaystyle{\frac{f^2}{2}}+fg\right)\right)\partial_x^k (f+g)
- \left(\mu f+g\right) \partial_x \sigma(\Gamma)\right)\right.\right.\\[5pt]
  &\quad\qquad  + \mu \Phi^{\prime\prime}(\Gamma)D\partial_x\Gamma\Bigg]  \Bigg\}\,dx
\end{align*}
for $k\in \{1,3\}$.
A tedious but  straight forward computation then yields
\begin{align}\label{EF}
\begin{split}
&\frac{d}{dt}\mathcal{E}_k(u)  = -\int_0^L \left\{ f\left[ \frac{f\partial_x^k((R_k+S_k\mu) f+S_k\mu g)}{\sqrt{3}}+\frac{\sqrt{3}}{2}\mu\left(S_k g \partial_x^k (f+g)-\partial_x \sigma(\Gamma)\right)\right]^2 \right. \\[5pt]
&\quad+\frac{1}{4}\mu^2 f\left[S_k g \partial_x^k (f+g)-\partial_x \sigma(\Gamma) \right]^2 +g\mu \left[\frac{S_k}{\sqrt{3}}g\partial_x^k (f+g)-\frac{\sqrt{3 }}{2}\partial_x \sigma(\Gamma)\right]^2\\[5pt]
&\quad  +\frac{ g \mu}{4}|\partial_x\sigma(\Gamma)|^2+\mu \Phi^{\prime\prime}(\Gamma)D|\partial_x \Gamma|^2 \Bigg\}\,dx.
\end{split}
\end{align}
\end{proof}

\section{Analysis of the gravity driven two--phase thin film flow with insoluble surfactant}\label{well posedness gravity}\label{Ch2}

In this section we prove a well--posedness and asymptotic stability result for the two--phase thin film equation with insoluble surfactant, where the motion of the fluids is driven by gravity only. The evolution of the thin--film flow is described by \eqref{system} if $k=1$, which is then a degenerated, strongly coupled parabolic system of second order in all three evolution equations.
Following the methods used in  \cite{EW1, EMM}, where local existence and asymptotic stability of strong solutions for systems modeling the evolution of a thin film with soluble surfactant and for systems describing a thin--film approximation of the two--phase Stokes problem, respectively, is shown, we prove analog results. 
In order to establish the well--posedness result, we assume the fluid with (strictly) higher density to be on the bottom ($\rho_1>\rho_2$)
and the initial data to satisfy  $(f^0,g^0,\Gamma^0)>0$.
We recall the gravity driven two--phase flow with insoluble surfactant \eqref{system} ($k=1$):
\begin{align} 
\nonumber
&\partial_t f =\partial_x \left[ f\left(\displaystyle{\frac{R_1f^2}{3}}\partial_x f + S_1\mu \left(\displaystyle{\frac{f^2}{3}} +\displaystyle{\frac{fg}{2}} \right) \partial_x(f+g) -\mu 	\displaystyle{\frac{f}{2}}\partial_x\sigma(\Gamma)\right)\right]\\[15pt]
\label{system1}
& \partial_t g = \partial_x \left[g\left( \displaystyle{\frac{R_1f^2}{2}} \partial_x f +S_1\left(  \displaystyle{\frac{g^2}{3}} +\mu \left( \displaystyle{\frac{f^2}{2}}+fg\right)\right)\partial_x (f+g) 
-\left(\mu f +\displaystyle{\frac{g}{2}} \right) \partial_x \sigma(\Gamma)\right)\right] \\[15pt]
\nonumber
&\partial_t \Gamma =\partial_x\left[\Gamma\left( \displaystyle{\frac{R_1f^2}{2}} \partial_x f +S_1\left(  \displaystyle{\frac{g^2}{2}}+\mu \left(\displaystyle{\frac{f^2}{2}}+fg\right)\right)\partial_x (f+g)
- \left(\mu f+g\right) \partial_x \sigma(\Gamma)\right)+ D\partial_x\Gamma\right]
\end{align}
for $t>0$ and $x\in (0,L)$, where 
\begin{align}\label{R1S1}
R_1=G_1-G_2\mu,\qquad S_1=G_2.
\end{align}
Furthermore, we assume Neumann boundary conditions to hold on the lateral boundary 
\begin{align*}
\begin{split}
&\partial_xf=\partial_xg=\partial_x\Gamma =0,\qquad \mbox{at}\quad x=0,L.
\end{split}
\end{align*}
We impose the following assumptions: 
\begin{itemize}
\item[G1) ] The density of the fluid on the bottom of the two--phase flow is higher than the density of the fluid on top, that is $\rho_1 > \rho_2$.
\end{itemize}
Assumption G1) in particular ensures, in view of \eqref{Gi} and \eqref{R1S1}, that
\begin{equation*}
R_1, S_1>0.
\end{equation*}
The surface tension, which depends on the surfactant concentration is assumed to be twice continuous differentiable and non--increasing
\begin{itemize}
\item[S1) ] $\sigma\in C^2(\R)$ and $-\sigma^\prime(s)\geq 0$ for all $s\geq 0$.
\end{itemize}
In order to study the well--posedness of the system of evolution equations \eqref{system1}, we need to find suitable spaces for solutions to work with and define 
\begin{align*}
&L_2:= L_2(0,L;\R^3),\\[5pt]
&H_N^2:= H_N^2(0,L;\R^3):= \{u\in H^2(0,L;\R^3)\mid \partial_x u(0)=\partial_x u(L) =0\}.
\end{align*}
The variable $u$ is to be seen as the triple $u=(f,g,\Gamma)$. Observe that we already incorporated the Neumann boundary condition in the space $H_N^2$.
For $\alpha\in[0,1]$ we define \footnote{We use the norm of $H^{2\alpha}(0,L;\R^3)\cap C([0,L];\R^3)$ to endow $U^\alpha$ with a metric.}
\[U^\alpha:= H_N^{2\alpha}(0,L;\R^3)\cap C([0,L];(0,\infty)^3),\]
where 
$$H_N^{2\alpha}:= H_N^{2\alpha}(0,L;\R^3):= \left\{ \begin{array}{lcl}\{u\in H^{2\alpha}(0,L;\R^3)\mid \partial_x u=0 \:\mbox{at}\: x=0,L\},\qquad &\mbox{if}&\; \alpha>\frac{3}{4},\\[5pt]
H^{2\alpha}(0,L;\R^3), &\mbox{if}&\; \alpha\in[0,\frac{3}{4}],\end{array}\right.$$
with $H^{2\alpha}(0,L;\R^3):= [L_2,H^2]_\alpha$ being the complex interpolation space between $H^2$ and $L_2$, called the \emph{Bessel potential space}. Let $\alpha>\frac{3}{4}$, then 
\begin{equation*}
H_N^{2\alpha} \subset C^1([0,L];\R^3)
\end{equation*}
and $U^\alpha\subset H_N^{2\alpha}$ is an open subset. For $u=(f,g,\Gamma)\in U^\alpha$ we define the diffusion matrix $a_G(u)$ to be
\begin{equation*}
\begin{pmatrix} (R_1+S_1\mu)\frac{f^3}{3}+G_2\mu\frac{f^2g}{2} & S_1\mu\left(\frac{f^3}{3}+\frac{f^2g}{2}\right) & -\mu\frac{f^2}{2}\sigma^\prime(\Gamma) \\
 S_1\frac{g^3}{3}+(R_1+S_1\mu)\frac{f^2g}{2}+S_1\mu fg^2 & G_2\frac{g^3}{3}+S_1\mu\left(\frac{f^2g}{2}+ fg^2\right) &-\left(\mu fg+\frac{g^2}{2}\right)\sigma^\prime(\Gamma) \\
\left( S_1\frac{g^2}{2}+(R_1+S_1\mu)\frac{f^2}{2}+G_2\mu fg\right)\Gamma  &\left(S_1\frac{g^2}{2}+S_1\mu \left(\frac{f^2}{2}+ fg\right)\right)\Gamma  &-\left(\mu f+g\right)\Gamma \sigma^\prime(\Gamma)+D \end{pmatrix}
\end{equation*}
and recast the problem $(\ref{system1})$ as an autonomous quasi--linear equation in the space $L^2$
\begin{equation}\label{A}
\partial_t u + A_G(u)u =0,\qquad t>0,\qquad u(0)=u^0,
\end{equation}
where the operator $A_G:U^\alpha\rightarrow \mathcal{L}(H_N^2,L_2)$ is given by
\begin{equation}\label{OA} A_G(u)w:=-\partial_x(a_G(u)\partial_xw),\qquad u\in U^\alpha,w\in H^2_N
\end{equation}
and $u^0=(f^0,g^0,\Gamma^0)$.

\subsection{Well--Posedness}
Studying the operator $A_G$ defined in $(\ref{OA})$, we prove that, assuming G1), S1) and $u^0\in U^\alpha$, there exists a unique, strictly positive solution on some time interval $[0,T)$, where $T\in (0,\infty]$ depends on the initial datum $u^0\in U^\alpha$. We claim that for fixed $u\in U^\alpha$, the linear operator $A_G(u)\in \mathcal{L}(H_N^2,L_2)$ is the negative generator of an analytic semigroup. Observe that the principal symbol of the linear operator $A_G(u)$, $u\in U^\alpha$, defined in \eqref{OA} is given by the matrix $a_G(u)$, which has positive eigenvalues in virtue of G1) and S1). 
It follows from \cite[Ex. 4.3.e)]{ANon} that $(A_G(u), B)$ is \emph{normally elliptic}, where $Bw=\partial_x w$ at $x=0,L$ for $w\in U^\alpha$. Taking into account that the coefficients of the matrix $a_G(u)$ are continuously differentiable and $A_G$ depends smoothly on its coefficients, \cite[Theorem 4.1]{ANon} implies that $-A_G(u)\in \mathcal{H}(H_N^2,L_2)$ and
\begin{equation}\label{AGH}
-A_G\in C^{1-}(U^\alpha,\mathcal{H}(H_N^2,L_2)).
\end{equation} 
With this, \cite[Theorem 12.1]{ANon} guarantees the following well--posedness result for \eqref{system1}:
\begin{thm}[Local Existence]\label{LE}
Let $\alpha\in (\frac{3}{4},1)$ and $u^0=(f^0,g^0,\Gamma^0)\in U^\alpha$. Assuming G1) and S1), the problem \eqref{A}
admits a unique positive strong solution
$$u=(f,g,\Gamma)\in C([0,T),U^\alpha)\cap C^\alpha([0,T),L_2)\cap C^1((0,T),L_2)\cap C((0,T),H_N^2)$$
with maximal time of existence $T\in(0,\infty]$. Moreover, $u$ depends continuously on the initial datum $u^0$ with respect to the topology of $U^\alpha$.
\end{thm}
Remark that Assumption G1) is crucial in order to obtain the well--posedness result. Hence, studying local strong solutions of \eqref{system1}, we need to exclude the case when $\rho_1=\rho_2$, that is, when both fluids have the same density but may differ in their viscous behavior. If $\rho_1=\rho_2$, then $R_1=0$ (cf. \eqref{Gi}) and it is easy to see that the matrix $a_G(u)$ has a zero eigenvalue. In this case we can no longer apply the theory in \cite{ANon}. 

\subsection{Asymptotic Stability}\label{2AS}
We show that the only steady states of \eqref{system1} are of the form where the films are flat and the surfactant concentration is uniformly distributed. Under the assumption that the surface tension is strictly decreasing, we obtain that that the steady states are asymptotically stable.
Similar as in \cite{EW2,EM13,EMM}, we use the energy functional \eqref{EnF}, which provides together with Assumption G1) that the set of steady states is determined by constants if the surface tension is strictly decreasing. Moreover, we show that if $u_*>0$ is a steady state, then it is asymptotically stable.
Considering system \eqref{system1}, it is clear that $u_*=(f_*,g_*,\Gamma_*)$, where $f_*,g_*$ and $\Gamma_*$ are positive constants, is an equilibrium. In order to determine all steady state solutions of \eqref{system1}, observe that a solution $u=(f,g,\Gamma)$ to \eqref{system1} as given by Theorem \ref{LE} satisfies \eqref{EF}.
Note that all terms on the right--hand side of the energy equality \eqref{EF} are non--positive. Hence, if $u=(f,g,\Gamma)$ is an equilibrium to \eqref{system1}, every single term on the right--hand side has to vanish, which implies that $\partial_x \sigma(\Gamma)= \partial_x(f+g)=\partial_x((R_1+S_1\mu) f + S_1\mu g)=0$. If $\sigma$ is strictly decreasing, we deduce that $f,g$ and $\Gamma$ are constant, in view of Assumption G1). 

\begin{cor} Suppose that $\sigma \in C^2(\R)$ is strictly decreasing and Assumption G1) is satisfied. Then, the only positive steady states to \eqref{system1} are of the form $(f_*,g_*,\Gamma_*)\in U^\alpha$ with constants $f_*,g_*,\Gamma_*>0$.
\end{cor}
In order to study the stability properties of these equilibria, we observe first, by a simple computation, that the mass of each fluid and the mass of surfactant is preserved by the evolution of the system, which is a consequence of the Neumann boundary conditions. 
\begin{lem}[Conservation of mass]\label{conservation}
Let $u=(f,g,\Gamma)$ be a solution to \eqref{system1} as found in Theorem \ref{LE}. Then, the mass of $u$ is preserved with time, that is, 
\begin{equation*}
\frac{d}{dt}\displaystyle{\int_0^L} f(t,x)\,dx =0\qquad \mbox{and}\qquad \frac{d}{dt} \displaystyle{\int_0^L} g(t,x)\,dx =0\qquad \mbox{and}\qquad \frac{d}{dt}\displaystyle{\int_0^L} \Gamma(t,x)\,dx =0
\end{equation*}
on $(0,T)$.
\end{lem}
The remainder of this section is dedicated to prove that, assuming the averaged initial surfactant concentration to be small, there exists for every initial data being close enough to the steady state a global positive strong solution to \eqref{system1} tending exponentially to the constant steady state.

Set $u_*=(f_*,g_*,\Gamma_*)$ with $f_*,g_*,\Gamma_*$ being positive constants and denote by 
\[\langle h \rangle:= \frac{1}{L}\int_0^Lh(x)\,dx  \]
the  average (with respect to space) of a function $h$.
Let $u=(f,g,\Gamma)$ be the unique strong solution to \eqref{system1} corresponding to the initial data $u^0=(f^0,g^0,\Gamma^0)\in U^\alpha$, satisfying $\langle f^0 \rangle = f_*, \langle g^0 \rangle =g_*$ and $\langle \Gamma^0 \rangle =\Gamma_* $. 
In order to study the stability property of the equilibrium $u_*$, we follow the ideas used in \cite{EW1, EMM} and eliminate the non--zero constant functions from the space we work in by introducing the projection $P\in \mathcal{L}(L_2)\cap\mathcal{L}(H_N^2)$, defined by 
$$Pu := u- \langle u \rangle= \left( f-\frac{1}{L}\int_0^Lf(x)\,dx, g-\frac{1}{L}\int_0^Lg(x)\,dx,\Gamma-\frac{1}{L}\int_0^L\Gamma(x)\,dx\right).$$
Clearly, $P$ defines a projection as
$$P^2u=PPu = P(u-\langle u \rangle)= Pu.$$
By means of the continuous projection we can decompose the spaces
\begin{align*} L_2 &= PL_2 \oplus (1-P)L_2, \\[5pt]
H_N^2 &= PH_N^2 \oplus (1-P)H_N^2
\end{align*}
into direct sums 
,
where $PL_2$, $PH_N^2$ contain the non--constant functions and the zero function in $L_2,H_N^2$ and $(1-P)L_2,(1-P)H_N^2$ contain the constant functions in $L_2,H_N^2$, respectively. Due to mass conservation (cf. Lemma \ref{conservation}) and continuity in $t=0$, a solution $u$ of \eqref{system1}, which satisfies initially $(1-P)u(0)=u_*$ fulfills $(1-P)u(t)=u_*$ as long as the solution exists.  Hence, we can decompose the solution $u$ with respect to the orthogonal sums:
$$u(t)= z(t)+u_* \in PL_2 \oplus (1-P)L_2,\qquad t\geq 0,$$
with $z(t)=Pu(t)$.
By $u$ being the corresponding solution to the initial data $u^0\in U^\alpha$, the function $z=u-u_*$ is a solution of 
\[
\partial_t z + A_G(z+u_*)z=0,\qquad z(0)=u^0-u_*.
\]
The stability property for $u_*$ is then equivalent to the one for the stationary solution $z=0$ of 
\begin{equation}\label{Z}
\partial_t z + A_G^*z= (A_G^*-A_G(z+u_*)\big|_{PH_N^2})z =:F(z),
\end{equation}
with $A_G^* w := A_G(u_*)w$ for $w\in PH_N^2.$
Due to the Neumann boundary conditions, both operators, $A_G^*$ and $[z\rightarrow A_G(z+u_*)z]$, map $PH_N^2$ into $PL_2$. Indeed, if $z\in PH_2$, then \[(1-P)A_G^*z  =\langle A_G^* z \rangle = -\frac{1}{L}\displaystyle{\int_0^L} \partial_x(a_G(u_*)\partial_x z)\,dx = 0\]
and
\[(1-P)A_G(z+u_*)z=\langle A_G(z+u_*)\rangle = -\frac{1}{L}\displaystyle{\int_0^L} \partial_x(a(z+u_*)\partial_x z)\,dx =0.\]
Note that, in view of $PH_N^2$ being continuously embedded into $PL_2$, the set $PH_N^2$ is an open neighborhood of zero in $PL_2$. Furthermore,
\begin{equation}\label{FZ} F\in C^1(PH_N^2,PL_2)\quad \mbox{with}\quad F(0)=F^\prime(0)=0,
\end{equation} where $F^\prime$ denotes the Fr\'echet derivative of $F$.

\begin{lem}\label{I} The operator $A_G^*:PH_N^2\subset PL_2 \rightarrow PL_2$ belongs to $\mathcal{H}(PH_N^2,PL_2)$, that is, $-A_G^*$ is the generator of an analytic semigroup on $PL_2$.
\end{lem}

\begin{proof} We already know from \eqref{AGH} that $-A_G\in C^{1-}(U^\alpha,\mathcal{H}(H_N^2,L_2))$, hence $-A_G(u_*)\in\mathcal{H}(H_N^2,L_2)$. By means of the orthogonal projection $P$ we can represent $-A_G(u_*)$ as a matrix operator
\[-A_G(u_*)= \begin{pmatrix} -A_G(u_*)\big|_{PH_N^2}& 0 \\0 & 0 \end{pmatrix}\in \mathcal{H}(PH_N^2\oplus(1-P)H_N^2,PL_2\oplus(1-P)L_2).\]
Because $A_G(u_*)(1-P)w=0$, the second column of the matrix has zero entries. Moreover $(1-P)A_G(u_*)w=\langle A_G(u_*)w \rangle=-\frac{1}{L}\int_0^L\partial_x(a_G(u_*)\partial_xw)=0$ for $w\in H_N^2$, which justifies the zero in the first entry of the second row. It follows from \cite[Theorem I.1.6.3]{ALin} that
\[-A_G(u_*)\big|_{PH_N^2} \in \mathcal{H}(PH_N^2,PL_2).\]
\end{proof}
In order to prove asymptotic stability for the equilibrium $z=0$ of $(\ref{Z})$, we apply the \emph{principle of linearized stability} (cf. \cite[9.1.1]{Lun}). For this purpose we state the following lemma:
\begin{lem}\label{K} Suppose $\sigma\in C^2(\R)$ is strictly decreasing and Assumption G1) is satisfied. Then there are numbers $\e, \omega_0>0$ such that the spectrum $\spec (-A_G^*)$ of $-A_G^*$ is contained in the half plane $[\re z \leq -\omega_0]$ provided that $0\leq\Gamma_*<\e$.
\end{lem}

\begin{proof} Take $w^0 = (f^0,g^0,\Gamma^0)\in PL_2$ arbitrary and let $w(t):= e^{-tA_G^*}w^0,\:t\geq 0$, be the unique strong solution in $PL_2$ to the linearized problem 
\begin{equation}\label{LP}\partial_t w +A_G^* w =0,\qquad t>0,\qquad w(0)=w^0.\end{equation}
By definition of $A_G^*=A_G(u_*)$, the function $w=(f,g,\Gamma)\in PH^2_N$ satisfies 
\[\partial_t \begin{pmatrix}\frac{S_1\mu}{R_1}( f+g )\\ f \\ z \Gamma \end{pmatrix}-\partial_x \left( \tilde{a}_G^z(u_*)\partial_x \begin{pmatrix} f+g \\ f \\ \Gamma \end{pmatrix}\right)=0,\]
where $z>0$ is a constant and the matrix $\tilde{a}_G^z(u_*)$ is given by
\[\begin{pmatrix} d_1 & S_1\mu \left(\frac{f_*^3}{3}+\frac{f_*^2g_*}{2}\right)  & -\frac{S_1\mu}{R_1}\left(\mu \frac{f^2}{2}+\mu f_*g_*+\frac{g_*^2}{2}\right)\sigma^\prime(\Gamma_*)\\
 S_1\mu\left(\frac{f_*^3}{3}+\frac{f_*^2g_*}{2}\right) & R_1\frac{f_*^3}{3} & -\mu\frac{f_*^2}{2}\sigma^\prime(\Gamma_*)\\
z\left(S_1 \frac{g_*^2}{2}+S_1\mu\left(\frac{f_*^2}{2}+f_*g_*\right)\right)\Gamma_* & zR_1\frac{f_*^2}{2}\Gamma_* & d_3
\end{pmatrix},
\]
where
\begin{align*}
d_1&:=\frac{S_1\mu}{R_1}\left(S_1 \frac{g_*^3}{3}+S_1\mu\left(\frac{f_*^3}{3}+ f_*^2g_*+f_*g_*^2\right)\right), \\[5pt]
d_3&:=-z(\mu f_*+g_*)\Gamma_*\sigma^\prime(\Gamma_*)+zD.
\end{align*}
Introducing the to $\tilde{a}_G^z(u_*)$ corresponding symmetric matrix 
\begin{equation*}
\hspace{-0.4cm}b_G^z(u_*):=\begin{pmatrix}
d_1 & S_1\mu \left(\frac{f_*^3}{3}+\frac{f_*^2g_*}{2}\right)  & j\\
 S_1\mu\left(\frac{f_*^3}{3}+\frac{f_*^2g_*}{2}\right) & R_1\frac{f_*^3}{3} & -\frac{1}{2}\left(\mu\frac{f_*^2}{2}\sigma^\prime(\Gamma_*)-zR_1\frac{f_*^2}{2}\Gamma_*\right)\\
j & -\frac{1}{2}\left(\mu\frac{f_*^2}{2}\sigma^\prime(\Gamma_*)-zR_1\frac{f_*^2}{2}\Gamma_*\right) & d_3 
\end{pmatrix}.
\end{equation*}
with
\begin{align*}
j&:=-\frac{1}{2}\left(\frac{S_1\mu}{R_1}\left(\mu \frac{f_*^2}{2}+\mu f_*g_*+\frac{g_*^2}{2}\right)\sigma^\prime(\Gamma_*)-z\left(S_1 \frac{g_*^2}{2}+S_1\mu\left(\frac{f_*^2}{2}+f_*g_*\right)\right)\Gamma_*\right),
\end{align*}
 we obtain that
\begin{equation*}
\frac{1}{2}\frac{d}{dt}\left(\frac{S_1\mu}{R_1}\|f+g\|_2^2+\|f\|_2^2+z\|\Gamma\|_2^2\right)+\left(b_G^z(u_*)\partial_x \begin{pmatrix} f+g\\ f \\ \Gamma \end{pmatrix} \left|\partial_x \begin{pmatrix} f+g\\ f \\ \Gamma \end{pmatrix} \right.  \right)_{2}=0.
\end{equation*}
If $\Gamma_*=0$, the matrix 
\begin{equation*}
b_G^z(f_*,g_*,0)=\begin{pmatrix}
\frac{S_1\mu}{R_1}\left(S_1 \frac{g_*^3}{3}+S_1\mu\left(\frac{f_*^3}{3}+ f_*^2g_*+f_*g_*^2\right)\right) & S_1\mu \left(\frac{f_*^3}{3}+\frac{f_*^2g_*}{2}\right)  & -\frac{\mu}{2}\frac{f_*^2}{2}\sigma^\prime(0)\\
 S_1\mu\left(\frac{f_*^3}{3}+\frac{f_*^2g_*}{2}\right) & R_1\frac{f_*^3}{3} & -\frac{\mu}{2}\frac{f_*^2}{2}\sigma^\prime(0)\\
-\frac{\mu}{2}\frac{f_*^2}{2}\sigma^\prime(0) & -\frac{\mu}{2}\frac{f_*^2}{2}\sigma^\prime(0) & zD 
\end{pmatrix}
\end{equation*}
is positive definite for some sufficient large constant $z>0$, since then all principal minors are  positive.
Hence, there exists $z>0$ and  $\e =\e(f_*,g_*)>0$ such that for $0\leq\Gamma_*<\e$ the matrix  $b_G^z(f_*,g_*,\Gamma_*)$ is positive definite and we deduce that
\[\frac{1}{2}\frac{d}{dt}\left(\frac{S_1\mu}{R_1}\|f+g\|_2^2+\|f\|_2^2+z\|\Gamma\|_2^2\right) \leq  -\eta \left\Vert \partial_x \begin{pmatrix} f+g\\ f \\ \Gamma \end{pmatrix}  \right\Vert_{2}^2\]
for some positive constant $\eta>0$. Recall that the average value of $\tilde{w}:=(f+g,f,\Gamma)$ where $(f,g,\Gamma)\in PH_N^2$ is given by  $\langle \tilde{w} \rangle=0$. Hence, there exists, by Poincar\'e's inequality, a constant $c>0$ such that $\|\tilde{w}\|_2^2\leq c^{-1}\|\partial_x \tilde{w}\|_2^2$ and it follows that
\[\frac{1}{2}\frac{d}{dt}\left(\frac{S_1\mu}{R_1}\|f+g\|_2^2+\|f\|_2^2+z\|\Gamma\|_2^2\right) \leq  -\eta c (\|f+g\|_2^2+\|f\|_2^2+\|\Gamma\|_2^2).\]
Set $m:=\mbox{max}\left\{\frac{S_1\mu}{R_1},z,1 \right\}$, then
\begin{equation}\label{ddt}
\begin{array}{lcl}
\displaystyle{\frac{1}{2}}\frac{d}{dt}\left(\frac{S_1\mu}{R_1}\|f+g\|_2^2+\|f\|_2^2+z\|\Gamma\|_2^2\right) &\leq &  -\frac{\eta c}{m} (m\|f+g\|_2^2+m\|f\|_2^2+m\|\Gamma\|_2^2) \\[5pt]
&\leq & -\frac{\eta c}{m} \left(\displaystyle{\frac{S_1\mu}{R_1}}\|f+g\|_2^2+\|f\|_2^2+z\|\Gamma\|_2^2\right).
\end{array}
\end{equation}
We will show that $\tilde{w}=(f+g,f,\Gamma)$, where $(f,g,\Gamma)$ is a solution to $(\ref{LP})$, has exponential decay, which implies that also $w=(f,g,\Gamma)$ is exponentially decreasing.
Observe that for $\tilde{w}=(f+g,f,\Gamma)\in PL_2$
\[\vert\vert\vert \tilde{w} \vert\vert\vert_2:= \left(\frac{S_1\mu}{R_1}\|f+g\|_2^2+\|f\|_2^2+z\|\Gamma\|_2^2\right)^{\frac{1}{2}}\]
defines an equivalent norm on $PL_2$. 
In virtue of $(\ref{ddt})$, we deduce that $\frac{d}{dt} \vert\vert\vert\tilde{w}\vert\vert\vert_2^2 \leq -C  \vert\vert\vert\tilde{w}\vert\vert\vert_2^2$ with $C:=2\frac{\eta c}{m}>0$. Hence, 
\begin{equation}\label{An} \vert\vert\vert\tilde{w}\vert\vert\vert_2 \leq e^{-t\frac{C}{2}} \vert\vert\vert\tilde{w}^0\vert\vert\vert_2,\qquad t\geq 0,\qquad \tilde{w}^0=(f^0+g^0,f^0,\Gamma^0).
\end{equation}
By equivalence of the norms $\vert\vert\vert\cdot \vert\vert\vert_2$ and $\|\cdot\|_2$, we obtain that $\|\tilde{w}\|_2\leq \tilde{c}e^{-t\frac{C}{2}} \|\tilde{w}^0\|_2$ for some constant $\tilde{c}>0$,
which means that $\tilde{w}$ has exponential decay. Therefore, also $w$ has exponential decay and
$$\|w(t)\|_2 = \|e^{-tA_*}w^0\|_2\leq Me^{-t\omega_0}\|w^0\|_2,\qquad t\geq 0,$$ for some $M\geq 1$ and $\omega_0 >0$. We deduce that $\spec (-A_G^*)\subset [\re z \leq -\omega_0]$.
\end{proof}

Combining Lemma \ref{I}, Lemma \ref{K} and $(\ref{FZ})$, we apply \cite[Theorem 9.1.2]{Lun} and arrive at the following asymptotic stability result for steady states of \eqref{system1}:

\begin{thm}[Asymptotic Stability]
Let $\alpha\in (\frac{3}{4},1)$, $\sigma\in C^2(\mathbb{R})$ be strictly decreasing and Assumption G1) be satisfied. Further let $f_*,g_*>0$ be arbitrary. Then there exist numbers $\varepsilon=\varepsilon(f_*,g_*)>0$, $\omega>0$ and $M\geq 1$, such that for $0\leq \Gamma_*<\varepsilon$ and any initial data $u^0=(f^0,g^0,\Gamma^0)\in H^2_N$ with $\langle f^0 \rangle = f_*$, $\langle g^0 \rangle = g_*$ and $\langle \Gamma^0 \rangle = \Gamma_*$ satisfying the smallness condition
$\|u^0-u_*\|_{H^2}\leq\varepsilon,$
there exists a unique global positive solution $$f,g,\Gamma\in C([0,\infty),U^\alpha)\cap C^\alpha([0,\infty),L_2)\cap C^1((0,\infty),L_2)\cap C((0,\infty),H_N^2)$$ to \eqref{system1}. The solution satisfies
$$\|u(t)-u_*\|_{H^2}+\|\partial_t u(t)\|_2\leq Me^{-\omega t}\|u^0-u_*\|_{H^2}\qquad \mbox{for}\quad t\geq 0,$$
where $u_*=(f_*,g_*,\Gamma_*)$.
\end{thm}

\section{Analysis of the capillary driven two--phase thin film flow with insoluble surfactant}\label{well posedness capillary}
 
 This section is devoted to study the two--phase thin film equation equipped with insoluble surfactant,  where capillary effects serve as the only driving force. The evolution of the two--phase flow is  described by \eqref{system} when $k=3$, which is then a degenerated, strongly coupled system of fourth order. Analogously to the previous section, we prove a well--posedness and asymptotic stability result. 
It occurs in particular one major difference in treating the fourth--order system  with regard to the second--order system studied in Section \ref{Ch2}. Observe that \eqref{system} in the case $k=3$ is of fourth order in the evolution equations for the two film heights and only of second order in the evolution equation for the surfactant concentration, which is strongly coupled to the fourth--order equations. Translating \eqref{system} into an abstract setting, the appearing matrix operator is of mixed order. The strong coupling of evolution equations of different orders courses difficulties in studying the matrix operator. Still, demanding a smallness condition on the surfactant concentration, we are able to show, by a perturbation argument, that the matrix operator is a generator of an analytic semigroup,
so that as before \cite[Theorem 12.1]{ANon} implies the well--posedness.
We will see that, in contrary to Theorem \ref{LE}, which states the well--posedness for the gravity driven two--phase thin film flow, considering the two--phase thin film with insoluble surfactant, where capillary effects are the only driving force, we do not need any assumption on the density of the two fluids. As before, the energy functional in Lemma \ref{EnF} provides that the set of steady states is determined by constants. Studying stability properties of these steady states then is similar to the analysis in the previous section. 

Recall the system of evolution equations given in \eqref{system} ($k=3$):
\begin{align}
\nonumber
&\partial_t f =\partial_x \left[ f\left(\displaystyle{\frac{R_3f^2}{3}}\partial_x^3 f + S_3\mu \left(\displaystyle{\frac{f^2}{3}} +\displaystyle{\frac{fg}{2}} \right) \partial_x^3(f+g) -\mu 	\displaystyle{\frac{f}{2}}\partial_x\sigma(\Gamma)\right)\right],\\[15pt]
\label{system2}
& \partial_t g = \partial_x \left[g\left( \displaystyle{\frac{R_3f^2}{2}} \partial_x^3 f +S_3\left(  \displaystyle{\frac{g^2}{3}} +\mu \left( \displaystyle{\frac{f^2}{2}}+fg\right)\right)\partial_x^3 (f+g) 
-\left(\mu f +\displaystyle{\frac{g}{2}} \right) \partial_x \sigma(\Gamma)\right)\right], \\[15pt]
\nonumber
&\partial_t \Gamma =\partial_x\left[\Gamma\left( \displaystyle{\frac{R_3f^2}{2}} \partial_x^3 f +S_3\left(  \displaystyle{\frac{g^2}{2}}+\mu \left(\displaystyle{\frac{f^2}{2}}+fg\right)\right)\partial_x^3 (f+g)
- \left(\mu f+g\right) \partial_x \sigma(\Gamma)\right)+ D\partial_x\Gamma\right]
\end{align}
for $t>0$ and $x\in (0,L)$ with initial data at $t=0$
\begin{equation*}
f(0,\cdot)=f^0,\quad g(0,\cdot)=g^0, \quad \Gamma(0,\cdot)=\Gamma^0
\end{equation*}
and boundary conditions
\begin{equation*}
\begin{array}{lll}
&\partial_xf=\partial_xg=\partial_x\Gamma =0,\\[10pt]
&\partial_x^3f=\partial_x^3g=0
\end{array}
\end{equation*}
at $x=0,L$. 
Recall, that the material constants $R_3$ and $S_3$ are negative and given by
\begin{equation*}
R_3:=-\sigma_1^c,\qquad S_3:= -\sigma_2^c.
\end{equation*}
We impose the following assumptions:
Given the surface tension coefficients  $\sigma_1^c\geq 0$ and $\sigma_2$ of the form
\[\sigma_2(\Gamma)= \sigma_2^c+\sigma(\Gamma),\]
we assume that  the part of the surface tension,  which depends on $\Gamma$, is non--increasing and the part of the surface tension, which is independent of the concentration of surfactant, is strictly positive, that is,
\begin{itemize}
\item[S1) ] $\sigma\in C^2(\R)$ and $-\sigma^\prime(s)\geq0$ for all $s\geq 0$,
\item[S2) ] $\sigma_1^c, \sigma_2^c >0$.
\end{itemize}

\subsection{Well--Posedness}\label{K31}

We need to define suitable spaces for the well--posedness of the system of evolution equations \eqref{system2}. Given $k\in \N$ and $n \in \N$, we set in the sequel
\begin{equation*}
H_B^k(0,L;\R^n):= \{u\in H^k(0,L;\R^n)\mid \partial_x^{2l+1}u(0)=\partial_x^{2l+1}u(L)=0\; \mbox{for all}\; l\in \N \;\mbox{with}\; 2l+2\leq k\}.
\end{equation*} 
These spaces are well defined by the Sobolev Embedding Theorem and endowed with the usual Sobolev norms. Since the system we are analyzing features both, second-- and fourth--order derivatives, the space $H_B^4(0,L;\R^2)\times H_B^2(0,L;\R) $ will play an important role. 
For $\alpha\in[0,1]$ and $\e>0$ we define \footnote{We use the norm of $\left( H^{4\alpha}(0,L;\R^2)\times H^{2\alpha}(0,L;\R) \right)\cap C([0,L],\R^3)$ to endow $O^\alpha$ with a metric. }
\begin{align*}
&O^\alpha:= \left( H_B^{4\alpha}(0,L;\R^2)\times H_B^{2\alpha}(0,L;\R) \right)\cap C([0,L],(0,\infty)^3), \\[5pt]
&O^\alpha_\e:=O^\alpha \cap  \{ u=(f,g,\Gamma) \in  H_B^{4\alpha}(0,L;\R^2)\times H_B^{2\alpha}(0,L;\R) \mid \|\Gamma\|_{H^{2\alpha}}<\e \},
\end{align*}
where  
$$ H_B^{s}(0,L;\R^n):= \left\{ \begin{array}{lcl}\{u\in H^{s}(0,L;\R^n)\mid \partial_x u=\partial_x^3u=0 \:\mbox{at}\: x=0,L\},\qquad &\mbox{if}&\; s\in (\frac{7}{2},4],\\[5pt]
\{u\in H^{s}(0,L;\R^n)\mid \partial_x u=0 \:\mbox{at}\: x=0,L\},\qquad &\mbox{if}&\; s\in(\frac{3}{2},\frac{7}{2}],\\[5pt]
H^{s}(0,L;\R^n), &\mbox{if}&\; s\in[0,\frac{3}{2}]
\end{array}\right.$$
with $H^{s}(0,L;\R^n)$ being the \emph{Bessel potential space} for $s\in [0,4]$. The product space  
\[ H_B^{4\alpha}(0,L;\R^2)\times H_B^{2\alpha}(0,L;\R)\]
 is the complex interpolation space $[L_2(0,L;\R^2)\times L_2(0,L;\R), H_B^{4}(0,L;\R^2)\times H_B^{2}(0,L;\R)]_{\alpha}$ for $\alpha\in [0,1]$ between the product spaces $H_B^{4}(0,L;\R^2)\times H_B^{2}(0,L;\R)$ and $L_2(0,L;\R^2)\times L_2(0,L;\R)$.
If $\alpha>\frac{7}{8}$, then
\[H_B^{4\alpha}(0,L;\R^2)\times H_B^{2\alpha}(0,L;\R)\subset C^3([0,L];\R^2)\cap C^1([0,L];\R).\]
 Furthermore, $O^\alpha$ and $O^\alpha_\e$ are open subsets in $ H_B^{4\alpha}(0,L;\R^2)\times H_B^{2\alpha}(0,L;\R)$. Note that the boundary conditions as well as the positivity are already incorporated into the sets $O^\alpha$ and $O^\alpha_\e$. For each $u=(f,g,\Gamma)\in O^\alpha$ we define the matrix $a_c(u)$ as
\begin{equation*}
\begin{pmatrix} (R_3+S_3\mu)\frac{f^3}{3}+S_3\mu\frac{f^2g}{2}& S_3\mu\left(\frac{f^3}{3}+\frac{f^2g}{2}\right)  & -\mu\frac{f^2}{2}\sigma^\prime(\Gamma) \\
 S_3 \frac{g^3}{3}+(R_3+S_3\mu)\frac{f^2g}{2}+S_3\mu fg^2&  S_3 \frac{g^3}{3}+S_3\mu\left(\frac{f^2g}{2}+ fg^2\right)&-\left(\mu fg+\frac{g^2}{2}\right) \sigma^\prime(\Gamma) \\
 \left(S_3\frac{g^2}{2}+(R_3+S_3\mu)\frac{f^2}{2}+S_3\mu fg\right)\Gamma &\left(S_3\frac{g^2}{2}+S_3\mu\left(\frac{f^2}{2}+fg\right)\right)\Gamma &-(\mu f+g)\Gamma \sigma^\prime(\Gamma)+D \end{pmatrix}
\end{equation*}
and rewrite the problem \eqref{system2} as a quasi--linear equation in the space $L^2(0,L;\R^3)$
\begin{equation}\label{A2}
\partial_t u + A_c(u)u =0,\qquad t>0,\qquad u(0)=u^0,
\end{equation}
where $u^0=(f^0,g^0,\Gamma^0)$ and the operator $A_c: =O^\alpha\rightarrow \mathcal{L}(H_B^4(0,L;\R^2)\times H_B^2(0,L;\R), L_2(0,L;\R^3))$ is given by
\[A_c(u)w:=-\partial_x\left(a_c(u)\begin{pmatrix} \partial_x^3 \tilde{f} \\ \partial_x^3 \tilde{g} \\ \partial_x \tilde{\Gamma} \end{pmatrix} \right),\qquad \mbox{for} \quad u\in O^\alpha, w:=(\tilde{f},\tilde{g},\tilde{\Gamma})\in H_B^4(0,L;\R^2)\times H_B^2(0,L;\R).\]

Letting  $\alpha \in (\frac{7}{8},1)$, we prove that there exists $\e>0$, such that starting with an initial data $u^0\in O^\alpha_\e $ and under the Assumptions S1) and S2), there exists a unique, strong solution on some time interval $[0,T)$, where $T\in (0,\infty]$ depends on the initial datum $u^0\in O^\alpha_\e$. 

\begin{thm}[Local Existence]\label{4MT} Let $\alpha\in (\frac{7}{8},1)$, S1) and S2) be satisfied. Then, there exists  $\e >0$, such that given $u^0 = (f^0,g^0,\Gamma^0)\in O^\alpha_{\e}$, the problem \eqref{A2} possesses a unique maximal strong solution
\begin{align*}
(f,g,\Gamma)\in &C([0,T);O^\alpha_{\e})\cap C^\alpha([0,T); L_2(0,L;\R^3))\cap C((0,T);H_B^{4}(0,L;\R^2)\times H_B^{2}(0,L;\R))\\[5pt]
& \cap C^1((0,T);L_2(0,L;\R^3)), 
\end{align*}
with maximal time of existence $T\in (0,\infty]$. Moreover, $u=(f,g,\Gamma)$ depends  continuously on the initial datum $u^0$ with respect to the topology of $O^\alpha$.
\end{thm}
Set $E_0:= L_2(0,L;\R^3)$ and $E_1:=H_B^{4}(0,L;\R^2)\times H_B^{2}(0,L;\R)$.
 Furthermore let $E_\theta:=[E_1,E_0]_\theta$ be the complex interpolation space between $E_1$ and $E_0$ for $\theta \in [0,1]$.  With 
\[O(\e):= C([0,L];(0,\infty)^3)\cap \{ u=(f,g,\Gamma) \in H^{4}(0,L;\R^2)\times H^{2}(0,L;\R)\mid \|\Gamma\|_{H^{2\alpha}}<\e \},\] we identify $O^\alpha_{\e} = O(\e) \cap E_\alpha$. By taking into account that $A_c$ depends smoothly on its coefficients we obtain that
\begin{equation}\label{needed}
A_c \in C^{1-}(O^\alpha_{\e},\mathcal{L}(E_1,E_0)).
\end{equation}
We show that there exists $\e>0$, such that for fixed $u\in O^\alpha_{\e}$, the linear operator $A_c(u)\in \mathcal{L}(E_1,E_0)$ is the negative generator of an analytic semigroup. Then, Theorem \ref{4MT} is a consequence of \cite[Theorem 12.1]{ANon}.

\begin{thm}\label{Aanal} Let $\alpha\in (\frac{7}{8},1)$ and Assumption S1), S2) be satisfied. Then, there exists $\e >0$, such that given $u=(f,g,\Gamma)\in O^\alpha_{\e}$, the operator $-A_c(u)$ generates an analytic semigroup in $L_2(0,L;\R^3)$, that is
\[-A_c(u)\in \mathcal{H}(H_B^{4}(0,L;\R^2)\times H_B^{2}(0,L;\R);L_2(0,L;\R^3)).\]
\end{thm}

Set $A_c^0(u):=A_c(f,g,0) $ for $u\in O^\alpha$. The principal symbol of $A_c^0(u)$ is then given by
\begin{align*}
-a_c(f,g,0)=\begin{pmatrix} (R_3+S_3\mu)\frac{f^3}{3}+S_3\mu\frac{f^2g}{2}& S_3\mu\left(\frac{f^3}{3}+\frac{f^2g}{2}\right)  & -\mu\frac{f^2}{2}\sigma^\prime(0) \\
 S_3 \frac{g^3}{3}+(R_3+S_3\mu)\frac{f^2g}{2}+S_3\mu fg^2&  S_3 \frac{g^3}{3}+S_3\mu\left(\frac{f^2g}{2}+ fg^2\right)&-\left(\mu fg+\frac{g^2}{2}\right) \sigma^\prime(0) \\
0 &0 &D 
\end{pmatrix}.
\end{align*}
 It follows immediately from a result on matrix generators \cite[Theorem I.1.6.1]{ALin}, that $A_c^0(u)$ is the negative generator of an analytic semigroup, if
\begin{align*}
\left(A_{11}(f,g)\right)\begin{pmatrix}\tilde{f} \\ \tilde{g}\end{pmatrix}:=-\partial_x\left(\begin{pmatrix} (R_3+S_3)\frac{f^3}{3}+S_3\mu\frac{f^2g}{2} & S_3\mu\left(\frac{f^3}{3}+\frac{f^2g}{2}\right) \\ 
						 S_3\frac{g^3}{3}+(R_3+S_3)\frac{f^2g}{2}+S_3 fg^2 &  S_3\frac{g^3}{3}+S_3\mu\left(\frac{f^2g}{2}+ fg^2\right)
				\end{pmatrix}\partial_x ^3\begin{pmatrix}\tilde{f} \\ \tilde{g}\end{pmatrix}\right)
\end{align*}
and
$D\partial_x^2 $
are negative generators of analytic semigroups.
Then, by means of a perturbation argument, we obtain the existence of $\e>0$, such that $A_c(u), u\in O^\alpha_\e,$ is the negative generator of an analytic semigroup.

\begin{satz}\label{A11analytic} Let $\alpha\in (\frac{7}{8},1)$, S1) and S2) be satisfied. Then
\begin{itemize}
\item[i) ] $-A_{11}(f,g)\in \mathcal{H}(H_B^4(0,L;\R^2), L_2(0,L;\R^2))$ for all $(f,g)\in \{H_B^{4\alpha}(0,L;\R^2)\mid f,g>0\}$,
\item[ii) ] $-D\partial_x^2\in \mathcal{H}( H_B^2(0,L;\R), L_2(0,L;\R))$ .
\end{itemize}
\end{satz}

In view of $D>0$, the operator $-D\partial_x^2$ is strongly elliptic and it is already well known that a strongly elliptic second--order operator  is the negative generator of an analytic semigroup on $L_2(0,L;\R)$ (cf. e.g. \cite[Theorem 7.2.7]{Paz}). We are left to show part $i)$ of Proposition \ref{A11analytic}. 

Following the lines of the proof of \cite[Lemma 4.1]{EMM}, where a similar problem is investigated for the more general case $n\geq 1$ (here $\Omega=(0,L)\subset \R$),  we show Proposition \ref{A11analytic} i) by verifying the \emph{Lopatinskii--Shapiro condition} for the pair $(\A,\mathcal{B})$, where $\A :=A_{11}(X)Y=-\partial_x(\tilde{a}(X)\partial_x^3Y)$ for $X=(f,g)\in \{X\in H^{4\alpha}_B(0,L;\R^2)\mid X>0\}$ fixed and $Y\in H_B^4(0,L;\R^2)$ with 
 \begin{align*}
 \tilde{a}(X)=\begin{pmatrix} (R_3+S_3)\frac{f^3}{3}+S_3\mu\frac{f^2g}{2} & S_3\mu\left(\frac{f^3}{3}+\frac{f^2g}{2}\right) \\ 
						 S_3\frac{g^3}{3}+(R_3+S_3)\frac{f^2g}{2}+S_3 fg^2 & S_3 \frac{g^3}{3}+S_3\mu\left(\frac{f^2g}{2}+ fg^2\right)
				\end{pmatrix} \qquad \mbox{in}\; \Omega=(0,L)
 \end{align*}
  and $\mathcal{B}$ being the boundary operator $\mathcal{B}:=(\mathcal{B}_1,\mathcal{B}_2,\mathcal{B}_3,\mathcal{B}_4)$ with
\begin{align*}\mathcal{B}_1Y=(1,0)\partial_x Y,\qquad \mathcal{B}_2Y=(0,1)\partial_x Y,\qquad \mathcal{B}_3Y=(1,0)\partial_x^3 Y,\qquad \mathcal{B}_4Y=(0,1)\partial_x^3 Y   \end{align*}
 on $\partial \Omega =\{ 0,L \}$ for $Y \in H^4_B(0,L;\R^2)$. 
The associate principal symbols of $(\A,\mathcal{B})$ are given by
\begin{alignat*}{2}
  & a_\pi(x,\xi)= -\tilde{a}(X(x))|\xi|^4\qquad &&\mbox{for}\quad (x,\xi)\in [0,L]\times \R,\\[5pt]
  &b_\pi(x,\xi)= \left((1,0)\xi,(0,1)\xi,(1,0)\xi^3,(0,1)\xi^3\right)\qquad &&\mbox{for}\quad (x,\xi)\in \{0,L\}\times \R.
\end{alignat*}
 The operator $\A$ is normally elliptic, since
 \begin{align*}
 \spec(a_\pi(x,\xi))\subset [\re z>0]\qquad \mbox{for all} \quad (x,\xi)\in [0,L]\times \{\xi\in \R\mid |\xi|=1\},
 \end{align*}
 which can be easily verified  by observing that the principal minors of $-\tilde{a}(X)$ are positive, which implies that $-\tilde{a}(X)$ is positive definite \footnote{Recall, that $R_3,S_3<0$.}. The boundary operator $\mathcal{B}$ is said to satisfy the \emph{Lopatinskii--Shapiro condition} with respect to $\A$ if for each $(x,\xi)$ belonging to the tangent bundle $T(\partial \Omega)$ and $\lambda \in [\re z\geq 0]$ with $(\xi,\lambda)\neq 0$ the only exponentially decaying solution of the boundary value problem on the half--line
 \begin{align}\label{LScon}
 [\lambda+a_\pi(x,\xi+ i\partial_t)]u =0,\quad t>0,\qquad b_\pi(x,\xi+i\partial_t)u(0)=0
 \end{align}
 is the zero solution. Then, the boundary value problem $(\A,\mathcal{B})$ is \emph{normally elliptic} if $\A$ is normally elliptic and $\mathcal{B}$ satisfies \eqref{LScon}.
 Due to \cite[Remark 4.2 b)]{ANon} it is sufficient to verify the Lopatinskii--Shaprio condition \eqref{LScon} for $(\A,\mathcal{B})$ in order to prove that $\A$ is the negative generator of an analytic semigroup. Since $\Omega=(0,L)$ is a subset of an one--dimensional space, the boundary $\partial \Omega=\{0,L\}$ is of dimension zero, which implies that the tangent space at the boundary is zero. This simplifies the Lopatinskii--Shapiro condition \eqref{LScon} so that we are left to to show that for all $\lambda \in [\re z\geq 0]$ the only exponentially decaying solution of the boundary value problem on the half--line
 \begin{align}\label{LScon2}
 [\lambda+a_\pi(x,i\partial_t)]u =0,\quad t>0,\qquad b_\pi(x,i\partial_t)u(0)=0
 \end{align}
 is the zero solution. The argumentation in the sequel follows the lines in the proof of \cite[Lemma 4.1]{EMM} setting $\xi=0$. The boundary value problem \eqref{LScon2} is equivalent to
 \begin{align}\label{bvp}
 \left\{ \begin{array}{lcl}
 \lambda a^{11} u_1 +\lambda a^{12}u_2+u_1^{(4)}=0,\\[5pt]
 \lambda a^{21} u_1 +\lambda a^{22}u_2+u_2^{(4)}=0,
 \end{array}\right. \qquad t>0,
 \end{align}
 with initial conditions
 \begin{align*}
 u_1^\prime(0)=u_2^\prime(0)= u_1^{\prime\prime\prime}(0)=u_2^{\prime\prime\prime}(0)=0,
 \end{align*}
 where $u_i^{(k)}$ denotes the $k$th derivative of $u_i$, $i=1,2$ and the matrix $(a^{ij})_{1\leq i,j\leq 2}$ the inverse of $-\tilde{a}(X)$, which exists by $-\tilde{a}(X)$ being positive definite. 
 Since $\lambda\neq 0$, we can express $u_2$ in virtue of the first equation in \eqref{bvp} as
 \begin{align}\label{u2}
 u_2 =-\frac{1}{\lambda a^{12}}\left[ u_1^{(4)}+\lambda a_{11} u_1 \right], 
 \end{align}
so that
\begin{align*}
u_2^{(4)} =-\frac{1}{\lambda a^{12}}\left[ u_1^{(8)}+\lambda a_{11} u_1^{(4)} \right].
\end{align*}
By means of the above equations, we obtain from the second equation in \eqref{bvp} an $8$th--order ordinary differential equation for $u_1$:
\begin{equation}\label{8order}
u_1^{(8)}+\lambda [a^{11}+a^{22}]u_1^{(4)}+\lambda^2\left[a^{11}a^{22}-a^{12}a^{21}\right]u_1=0,\qquad t>0,
\end{equation}
with initial conditions
\begin{align}\label{iniconu1} 
u_1^\prime(0)= u_1^{\prime\prime\prime}(0)= u_1^{(5)}(0)=u_1^{(7)}(0)=0.
\end{align}
A general solution of \eqref{8order} is given by the polynomial
\begin{equation}\label{gensol}
u_1(t)= \displaystyle{\sum_{k=1}^8} c_k e^{\Lambda_k t},\qquad t\geq 0,
\end{equation}
where $\{\Lambda_k\in \C\mid k=1,\ldots,8\}$ are the roots of the characteristic polynomial
\begin{align*}
\Lambda^8+\lambda[a^{11}+a^{22}]\Lambda^4+\lambda^2\left[a^{11}a^{22}-a^{12}a^{21}\right]=0.
\end{align*}
A solution to the above equation of $8$th--order is given via
\begin{equation*}
\Lambda^4_{\pm}=  \frac{\lambda}{2}\left(-[a^{11}+a^{22}]\pm  \sqrt{ (a^{11}-a^{22})^2+4a^{12}a^{21} } \right)=: \lambda E_\pm,
\end{equation*}
with $ E_{\pm}<0$ and $E_+\neq E_-$.
Hence, the roots $\Lambda_k$ are given by
\begin{align*}
&\Lambda_{1/2}=\pm \frac{1}{\sqrt{2}}(1+i)\sqrt[4]{-E_+},\qquad \Lambda_{3/4}=\pm \frac{1}{\sqrt{2}}(1-i)\sqrt[4]{-E_+}\\[5pt]
&\Lambda_{5/6}=\pm \frac{1}{\sqrt{2}}(1+i)\sqrt[4]{-E_-},\qquad \Lambda_{7/8}=\pm \frac{1}{\sqrt{2}}(1-i)\sqrt[4]{-E_-}.
\end{align*}
Recall that $u_1$ is claimed to have exponential decay, which implies in virtue of $\re \Lambda_k>0$ for $k\in \{1,3,5,7\}$, that $c_1, c_3,c_5,c_7=0$. In view of \eqref{iniconu1} and \eqref{gensol} we deduce that
\begin{align*}
\begin{pmatrix} u_1^{(1)}(0) \\  u_1^{(3)}(0)\\ u_1^{(5)}(0)\\ u_1^{(7)}(0)\end{pmatrix} = \begin{pmatrix} \Lambda_2 &\Lambda_4 & \Lambda_6 & \Lambda_8 \\
																								\Lambda_2^3 &\Lambda_4^3 & \Lambda_6^3 & \Lambda_8^3 \\
																								\Lambda_2^5 &\Lambda_4^5 & \Lambda_6^5 & \Lambda_8^5 \\
																								\Lambda_2^7 &\Lambda_4^7 & \Lambda_6^7 & \Lambda_8^7 \\							
																				\end{pmatrix}
																				\begin{pmatrix}
																				c_2 \\ c_4 \\ c_6 \\ c_8
																				\end{pmatrix}=0.
\end{align*}
 Due to $E_+\neq E_-$ it is clear that $\Lambda_2\neq \Lambda_4\neq\Lambda_6\neq\Lambda_8$ and the determinant of the $4\times 4$ matrix above is different from zero. Hence, $c_2, c_4 , c_6 ,c_8=0$. This implies that $u_1=u_2=0$ in view of \eqref{u2} and \eqref{gensol}, which completes the proof. 
 Eventually, we conclude that, in virtue of Proposition \ref{A11analytic}, the operator $-A_c^0$ belongs to $\mathcal{H}(H_B^4(0,L,\R^2)\times H_B^2(0,L;\R), L_2(0,L;\R^3))$.

Hence, by means of a perturbation argument (cf. \cite[Theorem 1.3.1]{ALin} ), there exists $\e>0$, such that
\[-A_c(u) \qquad\mbox{belongs to} \quad \mathcal{H}(H_B^4(0,L,\R^2)\times H_B^2(0,L;\R), L_2(0,L;\R^3))\]
for all $u\in O^\alpha_\e$ and Theorem \ref{4MT} is a consequence of \eqref{needed} and \cite[Theorem 12.1]{ANon}.

\subsection{Asymptotic Stability}
We study the stability properties of equilibrium solutions to \eqref{system2}. Following the approach as in \cite{EW1, EMM}, the analysis is similar to the one applied in Section \ref{2AS}. In order to obtain that a local solution of \eqref{system2} as found in Theorem \ref{4MT} satisfies the energy equality \eqref{EF}, we need to improve the regularity. 

\begin{cor} The local solution $u$ found in Theorem \ref{4MT} admits  the regularity
\[u\in C^{\frac{5}{4}}((0,T);H_B^{1}(0,L;\R^2)\times H_B^{\frac{1}{2}}(0,L;\R)).\]
\end{cor}
\begin{proof} We follow the lines in \cite[Section 4.1]{EMM}.
Theorem \ref{4MT} provides that
\[u\in C((0,T);H_B^{4}(0,L;\R^2)\times H_B^{2}(0,L;\R)) \cap C^1((0,T);L_2(0,L;\R^3)).\]
By \cite[Proposition II.1.1.2]{ALin}, this implies that 
\[u\in C^{1-\theta}((0,T);H_B^{4\theta}(0,L;\R^2)\times H_B^{2\theta}(0,L;\R))\]
for $\theta \in [0,1]$.
For $\rho\in\left(\frac{3}{8},1\right)$, the Sobolev Embedding Theorem yields
\[u\in C^{1-\rho}((0,T);C^1([0,L],\R^2)\times C([0,L],\R)).\]
Since $A_c$ depends smoothly on its coefficients, we deduce from Theorem \ref{Aanal} that 
\[-A_c(u)\in C^{1-\rho}((0,T); \mathcal{H}(H_B^{4}(0,L;\R^2)\times H_B^{2}(0,L;\R))).\] 
Note that $w:=u$ solves the linear parabolic problem
\[\partial_t w +A_c(u)w=0,\qquad w(0)=u(0)=u^0.\]
By \cite[Theorem 10.1]{ANon}, the unique solution $w$ profits from the 'regularizing' effect for parabolic equations and we obtain, in view of $w=u$, that
\begin{equation}\label{regr}
u\in C^{1-\rho}((0,T);H_B^{4}(0,L;\R^2)\times H_B^{2}(0,L;\R))\cap C^{2-\rho}((0,T);L_2(0,L;\R^3)).
\end{equation}
Since $\rho\in\left(\frac{3}{8},1\right)$, \eqref{regr} yields in particular that $u\in C^{\frac{1}{2}}((0,T);H_B^{4}(0,L;\R^2)\times H_B^{2}(0,L;\R))\cap C^{\frac{3}{2}}((0,T);L_2(0,L;\R^3))$
and, by \cite[Proposition 1.1.5]{Lun},
\begin{align*}
&u\in C^{\frac{1}{2}}((0,T);H_B^{4}(0,L;\R^2)\times H_B^{2}(0,L;\R))\cap C^{\frac{3}{2}}((0,T);L_2(0,L;\R^3))\\[5pt]
&\qquad\qquad\qquad\subset C^{\frac{3}{2}-\delta}((0,T);H_B^{4\delta}(0,L;\R^2)\times H_B^{2\delta}(0,L;\R),L_2(0,L;\R^2)\times L_2(0,L;\R))
\end{align*}
for $\delta\in(0,1)$. Set $\delta=\frac{1}{4}$, then $u\in C^{\frac{5}{4}}((0,T);H_B^{1}(0,L;\R^2)\times H_B^{\frac{1}{2}}(0,L;\R))$.
\end{proof} 

The above Corollary allows to differentiate the energy functional in \eqref{EnF} with respect to time and we find \eqref{EF} satisfied for a solution $u$ given by Theorem \ref{4MT}. Modifying the arguments in Section \ref{2AS}, we infer that the only steady states are given by constants, provided that $\sigma^\prime <0$, and these equilibria are asymptotically stable.

\begin{thm}[Asymptotic Stability]\label{Asst} Let $\alpha \in (\frac{7}{8},1)$, $\sigma\in C^2(\mathbb{R})$ be strictly decreasing and Assumption S1), S2) be satisfied. Further, let $u_* =(f_*,g_*,\Gamma_*)$ be a positive steady state solution of \eqref{system2}. Then $f_*,g_*$ and $\Gamma_*$ are constant and  there exist numbers $\e_*=\e_*(f_*,g_*)>0, \omega>0$ and $M\geq 1$, such that for $0\leq \Gamma_*<\e_*$ and any initial data $u^0=(f^0,g^0,\Gamma^0)\in H_B^4(0,L,\R^2)\times H_B^2(0,L,\R)$ with $\langle f^0 \rangle = f_*$, $\langle g^0 \rangle = g_*$ and $\langle \Gamma^0 \rangle = \Gamma_*$ satisfying the smallness condition
\[ \|u^0-u_*\|_{H^4_B\times H^2_B}<\e_*,\]
the solution $u$ of \eqref{system2} found in Theorem \ref{4MT} exists globally and
\[\|u(t)-u_*\|_{H^4_B\times H^2_B}+\|\partial_t u(t)\|_2 \leq Me^{-\omega t}\|u^0-u_*\|_{H^4_B\times H^2_B}\qquad \mbox{for all}\quad t\geq 0.\]
\end{thm} 

 \section*{Acknowledgement}
I am grateful to Joachim Escher and Christoph Walker for proposing this topic of research and for various helpful discussions. This work was supported  by the Deutsche Forschungsgemeinschaft (DFG) (Graduiertenkolleg GRK 1463 Analysis, Geometry and Stringtheory).

\bibliographystyle{elsarticle-num}

\end{document}